\documentclass[12pt, reqno, a4paper]{amsart}
\usepackage{amsmath, amssymb, amsthm, amscd}
\usepackage{extarrows}
\usepackage[T2A, T1]{fontenc}
\usepackage{txfonts}	
\usepackage{eucal}
\usepackage[dvips]{color}
\usepackage{multicol}
\usepackage[all]{xy}		
\usepackage{graphicx}
\usepackage{color}
\usepackage{colordvi}
\usepackage{xspace}
\usepackage{tikz}
\usepackage{enumitem}
\usepackage{ulem}
\usepackage[colorlinks,final,backref=page,hyperindex]{hyperref}

\allowdisplaybreaks 

\newif\ifflabel\flabelfalse
\ifflabel
\else
	
\fi

\newtheorem{theorem}{Theorem}[section]
\newtheorem{lemma}[theorem]{Lemma}
\newtheorem{corollary}[theorem]{Corollary}
\newtheorem{proposition}[theorem]{Proposition}
\theoremstyle{definition}
\newtheorem{definition}[theorem]{Definition}
\newtheorem{example}[theorem]{Example}
\newtheorem{remark}[theorem]{Remark}

\newcommand{\End}{\mathrm{End}}
\newcommand{\Hom}{\mathrm{Hom}}
\newcommand{\id}{\mathrm{id}}

\newcommand{\ad}{\mathrm{ad}}

\newcommand{\courant}[1]{\left\llbracket #1\right\rrbracket }

\newcommand{\delete}[1]{}

\begin{document}
\title[Extended $\mathcal{O}$-operators and extended PYBE]{Post-Poisson algebras, extended $\mathcal{O}$-operators and extended Poisson Yang-Baxter equations}

\author[Y.~Lin]{Yuanchang Lin}
\address{School of Mathematics, North University of China, Taiyuan 030051, China}
\email{linyuanchang@mail.nankai.edu.cn}

\author[D.~Lu]{Dilei Lu$^\ast$}
\address{College of Applied Science, Beijing Information Science and Technology University, Beijing 100192, China}
\email{ludyray@bistu.edu.cn}

\thanks{$^*$Corresponding author.}



\date{\today}

\begin{abstract}
	This paper introduces the extended $\mathcal{O}$-operators on Poisson algebras as a natural generalization of ordinary $\mathcal{O}$-operators, together with the extended Poisson Yang-Baxter equations. 
	We show that $\mathcal{O}$-operators of  weight $\lambda$ on Poisson algebras give rise to post-Poisson algebras, whose operad are the trisuccessor of the operad of Poisson algebras, and that extended $\mathcal{O}$-operators induce new Poisson algebra structures on module spaces.
	Equivalent characterizations of extended $\mathcal{O}$-operators are obtained via the symmetrizer-antisymmetrizer decomposition.
	The generalized Poisson Yang-Baxter equations are also introduced, and their connections with coboundary Poisson bialgebras and extended $\mathcal{O}$-operators are established. 
	The tensor form of extended $\mathcal{O}$-operators leads to the notion of the extended Poisson Yang-Baxter equations, which generalizes the notion of the Poisson Yang-Baxter equations. 
	Finally, the relationships among extended $\mathcal{O}$-operators, the extended Poisson Yang-Baxter equations, and the Poisson Yang-Baxter equations are studied in the framework of quadratic Poisson algebras and semi-direct product Poisson algebras.
\end{abstract}

\maketitle

\tableofcontents

 
\section{Introduction}\label{sec:intro}

The purpose of this paper is to establish extended structures for Poisson algebras, namely extended $\mathcal{O}$-operators and the extended Poisson Yang-Baxter equations, following the frameworks of extended $\mathcal{O}$-operators on Lie algebras~\cite{bai2010nonabelian} and associative algebras~\cite{bai2012O}.

A Poisson algebra is a vector space equipped with a commutative associative multiplication and a Lie bracket satisfying the Leibniz rule.
Poisson algebras serve as fundamental structures in various areas of mathematics and mathematical physics, including Poisson geometry \cite{vaisman2012lectures, weinstein1977lectures}, classical and quantum mechanics \cite{arnol2013mathematical, dirac2013lectures, odzijewicz2011hamiltonian}, algebraic geometry \cite{ginzburg2004poisson, polishchuk1997algebraic}, quantization theory \cite{huebschmann1990poisson, kontsevich2003deformation}, integrable systems and  quantum groups \cite{chari1995guide, drinfeld1986quantum}.
The bialgebra theory for Poisson algebras was introduced in~\cite{ni2013poisson} as a Poisson analogue of Drinfeld's Lie bialgebras~\cite{drinfeld1983hamiltonian}.
It is known that the classical Yang-Baxter equation (CYBE) and the associative Yang-Baxter equation (AYBE) are the fundamental Yang-Baxter-type equations governing Lie bialgebras and antisymmetric infinitesimal bialgebras, respectively~\cite{bai2010nonabelian, bai2012O}.
Furthermore, in the theory of Poisson bialgebras, a solution of the Poisson Yang-Baxter equation (PYBE), which combines the CYBE together with the AYBE under an invariant condition, gives rise to a coboundary Poisson bialgebra~\cite{lin2026quasitriangular, ni2013poisson}.

The notion of $\mathcal{O}$-operators (also known as Rota-Baxter operators of weight zero in certain contexts) provides a systematic method to construct solutions of Yang-Baxter-type equations~\cite{bai2010nonabelian, bai2011generalizations,bai2012O}.
For Lie algebras, extended $\mathcal{O}$-operators were introduced in~\cite{bai2010nonabelian} in order to study double Lie algebra structures and nonabelian generalized Lax pairs.
It was shown that extended $\mathcal{O}$-operators on Lie algebras are generalizations of Rota-Baxter operators and ordinary $\mathcal{O}$-operators~\cite{bai2011generalizations,kupershmidt1999classical,semenov1983classical}, and the relationship between extended $\mathcal{O}$-operators and the (extended) classical Yang-Baxter equations was presented.
Similar results for associative algebras and Novikov algebras were obtained in~\cite{bai2012O, yu2026extended}. 
These developments motivate our investigation of analogous extended structures in the Poisson algebra context.

On the other hand, the operad of post-Poisson algebras is the trisuccessor of the Poisson operad~\cite{bai2013splitting}, which is isomorphic to the Manin black product of the Poisson operad and the post-Lie operad~\cite{bai2013splitting,vallette2007Maninp}. Consequently, there is a natural construction of post-Poisson algebras from Poisson algebras equipped with Rota-Baxter operators of weights \cite{ni2013poisson}. 
Dotsenko's monadic approach to PBW-type theorems facilitates the use of operadic techniques, previously employed in homotopical algebra primarily, and proves the Koszulness of the post-Poisson operad ~\cite{dotsenko2021endofunctors}. 
Post-Poisson algebras have been shown to be quasi-classical limits of tridendriform algebra deformations of  commutative tridendriform algebras, extending the classic result that Poisson algebras are quasi-classical limits of associative algebra deformations of commutative associative algebras~\cite{chen2024postpoisson,ospel2022polarization}. 
Collectively, these works bridge deformation theory and  operadic algebras, establishing post-Poisson algebras as a significant and active research area.
The present paper also establishes relationships between $\mathcal{O}$-operators and post-Poisson algebras.

In this paper, we demonstrate the feasibility of extending the aforementioned framework to Poisson algebras.
More precisely, we first recall the notion of $A$-module Poisson algebras, which simultaneously encode the action of $A$ on $V$ via the Lie bracket and the commutative multiplication, together with the Poisson algebra structure on $V$ itself.
From such data, we define $\mathcal{O}$-operators of weight $\lambda$ on Poisson algebras and show that every $\mathcal{O}$-operator of weight $\lambda$ gives rise to a post-Poisson algebra structure. 
We give a necessary and sufficient condition under which a Poisson algebra admits a post-Poisson algebra structure whose associated Poisson algebra is the Poisson algebra itself.

We then introduce extended $\mathcal{O}$-operators on Poisson algebras.
Such an operator consists of a linear map $T:V\to A$ together with an extension $S:V \to A$ which may satisfy certain balanced, $A$-invariant and equivalent conditions, such that the failure of $T$ to be an $\mathcal{O}$-operator is controlled by $S$ and the mass parameters.
 Moreover, in the present paper, $T$ and the extension $S$ are shown to be correspond to the skew-symmetric and symmetric part of solutions of the PYBE, respectively, via the symmetrizer-antisymmetrizer decomposition.
Through this operator perspective, we can characterize solutions of the PYBE whose symmetric part is nonzero and $(\ad, L)$-invariant, thereby generalizing the classical result that \(\mathcal{O}\)-operators characterize skew-symmetric solutions of the Poisson Yang-Baxter equation (i.e., those with zero symmetric part).  
We show that extended $\mathcal{O}$-operators induce new Poisson algebra structures on module spaces, and provide equivalent descriptions of them in terms of homomorphisms of Poisson algebras and ordinary $\mathcal{O}$-operators via the symmetrizer-antisymmetrizer decomposition. 
Moreover, the notion of generalized Poisson Yang-Baxter equations (GPYBE) is also introduced, which arises naturally in the context of coboundary Poisson bialgebras.
Their relationship with extended $\mathcal{O}$-operators is then established.

We investigate the tensor form of extended $\mathcal{O}$-operators, which leads to the notion of the extended Poisson Yang-Baxter equation (EPYBE) of mass, as a special case of the GPYBE under the invariant condition.
When the symmetric part of the solution of the EPYBE is $(\ad,L)$-invariant, we establish a correspondence between extended $\mathcal{O}$-operators and solutions of the EPYBE.
In particular, for a suitable choice of the mass parameter, this correspondence gives the characterization of solutions of the PYBE in terms of ordinary $\mathcal{O}$-operators. Then we further study the extended $\mathcal{O}$-operators and the EPYBE in the framework of quadratic Poisson algebras and semi-direct product Poisson algebra. For a quadratic Poisson algebra, the nondegenerate bilinear form yields an identification between the underlying Poisson algebras and its dual. We demonstrate that an extended $\mathcal{O}$-operator on this Poisson algebra associated to the adjoint module corresponds to a solution of the EPYBE via such a nondegenerate bilinear form. Furthermore, the semi-direct product Poisson algebra establishes a correspondence between extended $\mathcal{O}$-operators on Poisson algebras and solutions of the EPYBE, generalizing the relationship between $\mathcal{O}$-operators and solutions of the PYBE.

The paper is organized as follows.
In Section~\ref{sec:o}, we recall the notions of $A$-module Poisson algebras, post-Poisson algebras and $\mathcal{O}$-operators, and introduce extended $\mathcal{O}$-operators on Poisson algebras.
The relationship between $\mathcal{O}$-operators and post-Poisson algebras is established.
Moreover, we show that extended $\mathcal{O}$-operators induce new Poisson algebra structures, and give equivalent characterization of extended $\mathcal{O}$-operators via the symmetrizer-antisymmetrizer decomposition.
In section~\ref{sec:tensor}, we first focus on the GPYBE and their connections with extended $\mathcal{O}$-operators and coboundary Poisson bialgebras. Then we introduce the EPYBE, establish the equivalence between extended $\mathcal{O}$-operators and solutions of the EPYBE under certain conditions, and study the EPYBE on quadratic Poisson algebras and semi-direct product Poisson algebras.

Throughout this paper, we work over a base field $\mathbf{k}$ of characteristic $0$, and all vector spaces and algebras are assumed to be finite-dimensional.
We adopt the following conventions and notations.
\begin{enumerate}
	\item[(i)]
	   Let $(A, \diamond)$ be a vector space equipped with a binary operation $\diamond: A \otimes A \rightarrow A$.
	   Let $L_{\diamond}(a)$ and $R_{\diamond}(a)$ denote the left and right multiplication operators, that is
	   \begin{equation*}
		   L_{\diamond}(a) b = R_{\diamond}(b) a = a \diamond b, \;\; \forall a, b \in A.
	   \end{equation*}
	   We also simply denote them by $L(a)$ and $R(a)$ respectively without confusion.
	   In particular, if $(A, [\ ,\ ])$ is a Lie algebra, we let $\ad_{[,]}(a) = \ad(a)$ denote the adjoint operator, that is
	   \begin{equation*}
		   \ad_{[,]}(a) b = \ad(a) b = [a, b], \;\; \forall a, b \in A.
	   \end{equation*}

	\item[(ii)]
	   Let $V$ be a vector space.
	   Denote the flip operator by $\tau: V \otimes V \rightarrow V \otimes V$, which is defined by
	   \begin{equation*}
		   \tau(u \otimes v) = v \otimes u, \;\; \forall u, v \in V.
	   \end{equation*}

	\item[(iii)]
	   Let $(A, \diamond)$ be a vector space equipped with a binary operation $\diamond: A \otimes A \rightarrow A$.
	   Let $r = \sum_{i} a_i \otimes b_i \in A \otimes A$.
	   Set
	   \begin{align*}
		   r_{12} = \sum_{i} a_i \otimes b_i \otimes 1, \;\;
		   r_{13} = \sum_{i} a_i \otimes 1 \otimes b_i, \;\;
		   r_{23} = \sum_{i} 1 \otimes a_i \otimes b_i, \\
		   r_{21} = \sum_{i} b_i \otimes a_i \otimes 1, \;\;
		   r_{31} = \sum_{i} b_i \otimes 1 \otimes a_i, \;\;
		   r_{32} = \sum_{i} 1 \otimes b_i \otimes a_i,
	   \end{align*}
	   where $1$ is the unit if $(A, \diamond)$ is unital or a symbol playing a similar role as the unit for the non-unital cases.
	   Furthermore, define compound symbols such as $r_{12} \diamond r_{13}$ by
	   \begin{equation*}
		   r_{12} \diamond r_{13} = \sum_{i, j} a_i \diamond a_j \otimes b_i \otimes b_j.
	   \end{equation*}

	\item[(iv)]
	   Denote the standard pairing between the dual space $V^*$ and $V$ by
	   \begin{equation*}
		   \langle \ ,\ \rangle: V^* \times V \rightarrow \mathbf{k}, \;\; \langle f, v \rangle := f(v), \;\; \forall f \in V^*, \; v \in V.
	   \end{equation*}

	\item[(v)]
	   Let $V, W$ be two vector spaces and $T: V \rightarrow W$ be a linear map.
	   Denote the dual map by $T^*: W^* \rightarrow V^*$, which is defined by
	   \begin{equation*}
		   \langle T^*(w^*), v \rangle = \langle w^*, T(v) \rangle, \;\; \forall v \in V, \; w^* \in W^*.
	   \end{equation*}

	\item[(vi)]
	   Let $A, V$ be vector spaces.
	   For a linear map $\zeta: A \rightarrow \End_{\mathbf{k}}(V)$, define a linear map $\zeta^*: A \rightarrow \End_{\mathbf{k}}(V^*)$ by $\zeta^*(a) = (\zeta(a))^*$, or more explicitly,
	   \begin{equation*}
		   \langle \zeta^*(a)v^*, u\rangle = \langle v^*, \zeta(a) u\rangle, \;\; \forall a \in A, \; u \in V, \; v^* \in V^*.
	   \end{equation*}
\end{enumerate}

\section{\texorpdfstring{$A$}{A}-module Poisson algebras, post-Poisson algebras and extended \texorpdfstring{$\mathcal{O}$}{O}-operators}\label{sec:o}
In this section, we recall the notions of $A$-module Poisson algebras, post-Poisson algebras and $\mathcal{O}$-operators, and introduce the concept of extended $\mathcal{O}$-operators as generalization of $\mathcal{O}$-operators.
We then present the relationship between post-Poisson algebras and $\mathcal{O}$-operators.
Furthermore, we show that extended $\mathcal{O}$-operators induce new Poisson algebra structures and give an equivalent characterization of extended $\mathcal{O}$-operators under certain conditions.

\subsection{\texorpdfstring{$A$}{A}-module Poisson algebras}\label{ss:pmpa}

\begin{definition}
	A \textbf{Poisson algebra} is a triple $(A, [\,,\,], \cdot)$, where $(A, [\,,\,])$ is a Lie algebra and $(A, \cdot)$ is a commutative associative algebra satisfying
	\begin{equation*}
		[a, b \cdot c] = [a, b] \cdot c + b \cdot [a, c], \;\; \forall a, b, c \in A.
	\end{equation*}
\end{definition}

\begin{definition}\cite{ni2013poisson}
	Let $(A, [\,,\,], \cdot)$ be a Poisson algebra, $V$ be a vector space and $\rho, \zeta: A \rightarrow \End_{\mathbf{k}}(V)$ be two linear maps.
	Then $(V, \rho, \zeta)$ is called a {\bf module} of $(A, [\,,\,], \cdot)$ if $(V, \rho)$ is a module of $(A, [\,,\,])$ and $(V, \zeta)$ is a module of $(A, \cdot)$ such that
	\begin{equation}
		\rho(a \cdot b) = \zeta(b) \rho(a) + \zeta(a) \rho(b), \;\;
		\zeta([a, b]) = \rho(a) \zeta(b) - \zeta(b) \rho(a), \;\;
		\forall a, b \in A. \label{eq:repp}
	\end{equation}
	An {\bf $A$-module homomorphism} between modules $(V_{1}, \rho_{1}, \zeta_{1})$ and $(V_{2}, \rho_{2}, \zeta_{2})$ of $(A, [\,,\,], \cdot)$ is a linear map $\varphi: V_{1} \rightarrow V_{2}$ such that
	\begin{equation*}
		\varphi(\rho_{1}(a)v) = \rho_{2}(a)(\varphi(v)), \;\;
		\varphi(\zeta_{1}(a)v) = \zeta_{2}(a)(\varphi(v)), \;\;
		\forall a \in A, \; v \in V_{1}.
	\end{equation*}
\end{definition}

Note that if $(V, \rho, \zeta)$ is a module of $(A, [\,,\,], \cdot)$, then $(V^*, -\rho^*, \zeta^*)$ is also a module of $(A, [\,,\,], \cdot)$.

\begin{example}
	Let $(A, [\,,\,], \cdot)$ be a Poisson algebra.
	Then both $(A, \ad_{[,]}, L_{\cdot})$ and $(A^*, -\ad_{[,]}^*, L_{\cdot}^*)$ are modules of $(A, [\,,\,], \cdot)$, called the {\bf adjoint} and {\bf coadjoint} module of $(A, [\,,\,], \cdot)$, respectively. 
\end{example}

\begin{definition}[\cite{ni2013poisson}]\label{def:module}
	Let $(A, [\,,\,], \cdot)$ and $(V, [\,,\,]_V, \cdot_V)$ be Poisson algebras, and $\rho, \zeta: A \rightarrow \End_{\mathbf{k}}(V)$ be two linear maps.
	The quintuple $(V, [\,,\,]_V, \cdot_V, \rho, \zeta)$ is called an \textbf{$(A, [\,,\,], \cdot)$-module Poisson algebra} (and \textbf{$A$-module Poisson algebra} for short) if $(V, \rho, \zeta)$ is a module of $(A, [\,,\,], \cdot)$ and the following equalities hold for all $a \in A$ and $u, v \in V$:
	\begin{align}
		\rho(a)([u, v]_V)  & = [\rho(a)u, v]_V + [u, \rho(a)v]_V, \quad
		\zeta(a)(u \cdot_V v) = (\zeta(a)u) \cdot_V v, \label{eq:pmc}       \\
		\rho(a)(u \cdot_V v) & = (\rho(a)u) \cdot_V v + u \cdot_V (\rho(a)v), \;\;
		[u, \zeta(a)v]_V = -(\rho(a)u) \cdot_V v + \zeta(a)([u, v]_V). \label{eq:pmd}
	\end{align}
\end{definition}

\begin{proposition}[\cite{ni2013poisson}]
	Let $(A, [\,,\,], \cdot)$ be a Poisson algebra, $V$ be a vector space with binary operations $[\,,\,]_V$ and $\cdot_V$, and $\rho, \zeta: A \rightarrow \End_{\mathbf{k}}(V)$ be two linear maps.
	Define binary operations $\courant{\,,\,}$ and $\bullet$ on $A \oplus V$ by (for all $a, b \in A$ and $u, v \in V$)
	\begin{align*}
		\courant{a + u, b + v} & := [a, b] + \rho(a) v - \rho(b) u + [u, v]_V,     \\
		(a + u) \bullet (b + v) & := a \cdot b + \zeta(a) v + \zeta(b) u + u \cdot_V v.
	\end{align*}
	Then $(V, [\,,\,]_V, \cdot_V, \rho, \zeta)$ is an $A$-module Poisson algebra if and only if $(A \oplus V, \courant{\,,\,}, \bullet)$ is a Poisson algebra, which is called the \textbf{semi-direct product of $(A, [\,,\,], \cdot)$ and $(V, [\,,\,]_V, \cdot_V, \rho, \zeta)$} and denoted by $(A \ltimes_{\rho, \zeta} V, \courant{\,,\,}, \bullet)$.
\end{proposition}

\begin{example}
	Let $(A, [\,,\,], \cdot)$ be a Poisson algebra. Then
	\begin{enumerate}
		\item[(i)]
		   $(A, [\,,\,], \cdot, \ad_{[,]}, L_{\cdot})$ is an $A$-module Poisson algebra.

		\item[(ii)]
		   A module of $(A, [\,,\,], \cdot)$ is equivalent to an $A$-module Poisson algebra equipped with the trivial Poisson algebra structure.
		   In the sequel, we will adopt this perspective for convenience and always regard a module of $(A, [\,,\,], \cdot)$ as such an $A$-module Poisson algebra.
	\end{enumerate}
\end{example}

\subsection{\texorpdfstring{$\mathcal{O}$}{O}-operators and post-Poisson algebras}
Recall the notion of $\mathcal{O}$-operations, motivated by the corresponding notions in the context of Lie and associative algebras \cite{bai2010nonabelian, bai2012O}.
\begin{definition}\label{def:oo}
	Let $(A, [\,,\,], \cdot)$ be a Poisson algebra, $(V, [\,,\,]_V, \cdot_V, \rho, \zeta)$ be an $A$-module Poisson algebra and $\lambda \in \mathbf{k}$.
	A linear map $T: V \rightarrow A$ is called an {\bf $\mathcal{O}$-operator} of weight $\lambda$ on $(A, [\,,\,], \cdot)$ associated to $(V, [\,,\,]_V, \cdot_V, \rho, \zeta)$ if for all $u, v \in V$
	\begin{align*}
		[T(u),T(v)]   & = T\big( \rho(T(u))v - \rho(T(v))u + \lambda [u, v]_V \big),   \\
		T(u) \cdot T(v) & = T\big( \zeta(T(u))v + \zeta(T(v))u + \lambda u \cdot_V v \big).
	\end{align*}
	An $\mathcal{O}$-operator of weight $\lambda$ on $(A, [\,,\,], \cdot)$ associated to $(A, [\,,\,], \cdot, \ad_{[,]}, L_\cdot)$ is called a {\bf Rota-Baxter operator} of weight $\lambda$ on $(A, [\,,\,], \cdot)$.
\end{definition}

\begin{remark}
	When $\lambda \neq 0$, the following statements are equivalent.
	\begin{enumerate}
		\item[(i)]
		   $T$ is an $\mathcal{O}$-operator of weight $\lambda$ on $(A, [\,,\,], \cdot)$ associated to $(V, [\,,\,]_V, \cdot_V, \rho, \zeta)$.

		\item[(ii)]
		   $\frac{T}{\lambda}$ is an $\mathcal{O}$-operator of weight $1$ on $(A, [\,,\,], \cdot)$ associated to $(V, [\,,\,]_V, \cdot_V, \rho, \zeta)$.

		\item[(iii)]
		   $T$ is an $\mathcal{O}$-operator of weight $1$ on $(A, [\,,\,], \cdot)$ associated to $(V, \lambda [\,,\,]_V, \lambda \cdot_V, \rho, \zeta)$.
	\end{enumerate}
\end{remark}

We show that, as a generalization of the Rota-Baxter operator, an $\mathcal{O}$-operator is equivalent to a Rota-Baxter operator on a larger Poisson algebra.

 \begin{proposition}
	Let $(A, [\,,\,], \cdot)$ be a Poisson algebra, $(V, [\,,\,]_V, \cdot_V, \rho, \zeta)$ be an $A$-module Poisson algebra and $T: V \rightarrow A$ be a linear map. 
	Let $(\mathfrak{A}=A \ltimes_{\rho, \zeta} V, \courant{\,,\,}, \bullet)$ be the semi-direct product of $(A, [\,,\,], \cdot)$ and $(V, [\,,\,]_V, \cdot_V, \rho, \zeta)$. For a given $ \lambda \in \mathbf{k}$, the
following statements are equivalent.
\begin{enumerate}[label=(\roman*)]
	\item\label{Rla1} 
	$T$ is an $\mathcal{O}$-operator of weight $\lambda$ on $(A, [\,,\,], \cdot)$ associated to $(V, [\,,\!\,]_V, \cdot_V, \rho, \zeta)$.
	
	\item\label{Rla2}  
	The linear map $\hat{T}$ given by 
	\begin{equation*}
		\hat{T}: \mathfrak{A} \to \mathfrak{A},\quad a + u \mapsto -\lambda a + T(u), \quad \forall a \in A, u \in V,
	\end{equation*}
	is a Rota-Baxter operator of weight $\lambda$ on $(\mathfrak{A}, \courant{\,,\,}, \bullet)$.
	
	\item\label{Rla3} 	
	The linear map $ \check{T}:=-\lambda \id - \hat{T}$ given by 
	\begin{equation*}
	\check{T}: \mathfrak{A} \to \mathfrak{A},\quad a + u \mapsto -T(u) - \lambda u, \quad \forall a \in A, u \in V,
	\end{equation*} 
	is a Rota-Baxter operator of weight $\lambda$ on $(\mathfrak{A}, \courant{\,,\,}, \bullet)$.
\end{enumerate}	
 \end{proposition}
 \begin{proof}
	\ref{Rla1} $\Longleftrightarrow$ \ref{Rla2}. 
	Let $a,b \in A$ and $u,v \in V$. 
	Then we have
	\begin{align*}
		\courant{\hat{T}(a + u),\hat{T}(b + v)} &= [-\lambda a \!+\! T(u),-\lambda b \!+\! T(v)] =\lambda^2[a,b]  \!- \!\lambda [a,T(v)]  \! - \!\lambda [T(u),b]  \!+\! [T(u) , T(v)] ,\\
		\hat{T}\courant{ \hat{T}(a+u), b+v } &=  \hat{T}([T(u),b]  - \lambda[a,b]  - \lambda \rho(a)v + \rho(T(u))v)\\
		&= - \lambda[T(u),b]  + \lambda^2[a,b]  - \lambda T(\rho(a)v) + T(\rho(T(u))v),\\
		\hat{T}\courant{ a+u,\hat{T}(b+v) } &=  \hat{T}([a,T(v)]  - \lambda[a,b]  + \lambda \rho(b)u - \rho(T(v))u)\\
		&=  - \lambda[a,T(v)]  + \lambda^2[a,b]  + \lambda T(\rho(b)u) - T(\rho(T(v))u),\\
		\lambda\hat{T}\courant{a+u, b+v} &= -\lambda^2[a,b] + \lambda T(\rho(a)v ) - \lambda T(\rho(b)u ) + \lambda T([u,v]_V).
\end{align*}
Hence,
\begin{align*}
	&\courant{\hat{T}(a + u),\hat{T}(b+v)} - \hat{T} \big(\courant{ \hat{T}(a+u), b+v } + \courant{ a+u, \hat{T}(b+v) } + \lambda\courant{a+u, b+v} \big) \\
	&= [T(u), T(v)] - T(\rho(T(u)) v - \rho(T(v))u + \lambda [u, v]_V).
\end{align*}
Similarly, we also have
\begin{align*} 
	&\hat{T}(a+u) \bullet \hat{T}(b+v) - \hat{T}( \hat{T}(a+u)\bullet (b+v)  +  (a+u) \bullet \hat{T}(b+v)  + \lambda  (a+u)\bullet (b+v)) \\
	&=	T(u) \cdot T(v) - T\big( \zeta(T(u))v + \zeta(T(v))u + \lambda u \cdot_V v \big).
\end{align*}
Thus, $T$ is an $\mathcal{O}$-operator of weight $\lambda$ on $(A, [\,,\,], \cdot)$ associated to $(V, [\,,\!\,]_V, \cdot_V, \rho, \zeta)$ if and only if $\hat{T}$ is a Rota-Baxter operator of weight $\lambda$ on $(\mathfrak{A}, \courant{\,,\,}, \bullet)$.
 
\ref{Rla2} $\Longleftrightarrow$ \ref{Rla3}. 
This follows from the well-known fact that $P : A \to A$ is a Rota-Baxter operator of weight $\lambda$ on a Poisson algebra $(A, [\,,\,], \cdot)$ if and only if $-\lambda \id-P : A \to A$ is a Rota-Baxter operator of weight $\lambda$.
\end{proof}

Next we recall the notion of post-Poisson algebras, which are the Poisson analogue of the dendriform trialgebras~\cite{loday2004trialgebras}.

\begin{definition}[\cite{vallette2007homology}]
	A {\bf (left) post-Lie algebra} is triple $(A, \{\,,\,\}, \star)$, where $(A, \{\,,\,\})$ is a Lie algebra and $\star$ is a binary operation satisfying the following equalities for all $a, b, c \in A$:
	\begin{align}
		c \star \{a, b\} - \{c \star a, b\} - \{a, c \star b\}                           & = 0, \label{postLie1}\\
		 c \star (b \star a) - b \star (c \star a) + (b \star c) \star a - (c \star b) \star a + \{b, c\} \star a & = 0. \label{postLie2} 
	\end{align}
\end{definition}

\begin{definition}[\cite{loday2007algebra}]
	A {\bf (left) commutative dendriform trialgebra} is a triple $(A, \circ, \succ)$, where $(A, \circ)$ is a commutative associative algebra and $\succ$ is a binary operation satisfying the following equalities for all $a, b, c \in A$:
	\begin{align}
		(a \succ b + b \succ a + a \circ b) \succ c & = a \succ (b \succ c), \label{comtri1}\\
		(a \succ b) \circ c             & = a \succ (b \circ c). \label{comtri2}
	\end{align}
\end{definition}

\begin{definition}[\cite{ni2013poisson}]
	A \textbf{post-Poisson algebra} is a tuple $(A, \{\,,\,\}, \star, \circ, \succ)$, such that $(A, \{\,,\,\}, \circ)$ is a Poisson algebra, $(A, \{\,,\,\}, \star)$ is a post-Lie algebra, $(A, \circ, \succ)$ is a commutative dendriform trialgebra, and they are compatible in the sense that for all $a, b, c \in A$:
	\begin{align}
		\{a, b \succ c\}              & = b\succ \{a, c\} -  (b \star a) \circ c,             \label{postPois1} \\
		a \star (b \circ c)             & = (a \star b) \circ c + b \circ (a \star c),           \label{postPois2}  \\
		(a \succ b + b \succ a + a \circ b) \star c & = a \succ (b \star c) + b \succ (a \star c) ,           \label{postPois3}  \\
		    (a \star b - b \star a + \{a, b\}) \succ c       & = a \star (b \succ c) - b \succ (a \star c) .\label{postPois4} 
	\end{align}
\end{definition}

Recall that a \textbf{derivation} of a commutative dendriform trialgebra $(A, \circ, \succ)$ is defined to be a linear map $D:A \to A$ satisfying 
\begin{equation*}
	D(a \circ b) = D(a) \circ b + a \circ D(b),\;\;
	D(a \succ b) = D(a) \succ b + a\succ D(b), \;\; 
	\forall a, b \in A.
\end{equation*}

There is a construction of a post-Poisson algebra by a pair of commuting derivations on a commutative dendriform trialgebra.

\begin{proposition}\label{sdztodpp}
	Let $(A, \circ, \succ)$ be a commutative dendriform trialgebra and $D_1,D_2: A \to A$ be two commuting derivations on $(A,\circ, \succ)$. 
	Define new binary operations $\{\,,\,\}, \star: A \otimes A \to A$ on $A$ by
	\begin{align*}
		\{a,b\}= D_1(a)\circ D_2(b) - D_2(a) \circ D_1(b), \;\;
		a \star b = D_1(a)\succ D_2(b) - D_2(a) \succ D_1(b), \;\; \forall a, b \in A. 
	\end{align*} 
	Then $(A, \{\,,\,\}, \star)$ is a post-Lie algebra. Moreover, $(A, \{\,,\,\}, \star, \circ, \succ)$ is a post-Poisson algebra.
\end{proposition}
\begin{proof}
	It is known that $(A,\{\,,\,\})$ is a Lie algebra. 
	For all $a, b, c \in A$, we have
	\begin{align*}
		& c \star \{a, b\} - \{c \star a, b\} - \{a, c \star b\}\\
		&=   	D_1(c)\succ D_2 ( D_1(a)\circ D_2(b) - D_2(a) \circ D_1(b)) - D_2(c)\succ D_1 ( D_1(a)\circ D_2(b) - D_2(a) \circ D_1(b)) \\
		&\quad - D_1( D_1(c)\succ D_2(a) - D_2(c) \succ D_1(a) )  \circ D_2(b) + D_2( D_1(c)\succ D_2(a) - D_2(c) \succ D_1(a) )  \circ D_1(b)\\
		&\quad+ D_1( D_1(c)\succ D_2(b) - D_2(c) \succ D_1(b) )  \circ D_2(a) - D_2( D_1(c)\succ D_2(b) - D_2(c) \succ D_1(b) )  \circ D_1(a)\\
		&\overset{ \eqref{comtri2}}{=}0.  
	\end{align*}
	Hence, Eq.~\eqref{postLie1} holds.
	Similarly, Eq.~\eqref{postLie2} holds, so $(A, \{\,,\,\}, \star)$ is a post-Lie algebra. 
	Moreover
	\begin{align*}
		&\{a, b \succ c\} - b \succ \{a, c\} + c \circ (b \star a)\\
		&= D_1(a)\circ (D_2(b) \succ c)) +D_1(a)\circ (b \succ D_2( c)) - D_2(a) \circ (D_1(b) \succ c ) - D_2(a) \circ (b\succ D_1( c) )\\
		&\quad - b \succ (D_1(a)\circ D_2(c) - D_2(a) \circ D_1(c))+ c \circ (D_1(b)\succ D_2(a) - D_2(b) \succ D_1(a))\\
		&\overset{ \eqref{comtri2}}{=} D_1(a)\circ (D_2(b) \succ c)) - D_2(a) \circ (b \succ D_1( c) ) + b \succ ( D_2(a) \circ D_1(c)) - c \circ (D_2(b) \succ D_1(a))\\
		&\overset{\eqref{comtri2}}{=} 0. 
	\end{align*}
	Thus, Eq.~\eqref{postPois1} holds.
	The remaining Eqs.~\eqref{postPois2}-\eqref{postPois4} follow similarly.
	Therefore, $(A, \{\,,\,\}, \star, \circ, \succ)$ is a post-Poisson algebra.
\end{proof}

\begin{remark}
	Proposition~\ref{sdztodpp} provides a construction of post-Poisson algebras from a pair of commuting derivations on commutative dendriform trialgebras, analogous to the classical construction of Poisson algebras from a pair of commuting derivations on commutative associative algebras.
\end{remark}

\begin{proposition}[\cite{ni2013poisson}]\label{prop:pp2p}
	Let $(A, \{\,,\,\}, \star, \circ, \succ)$ be a post-Poisson algebra.
	Define two binary operations $[\,,\,], \cdot: A \otimes A \rightarrow A$ on $A$ by
	\begin{align}
		[a, b] :&= a \star b - b \star a + \{a, b\}, \label{assoLie} \\ 
		a \cdot b :&= a \succ b + b \succ a + a \circ b, \;\;
		\forall a, b \in A. \label{assoCom}
	\end{align}
	Then $(A, [\,,\,], \cdot)$ is a Poisson algebra, called the {\bf associated Poisson algebra} of $(A, \{\,,\,\}, \star, \circ, \succ)$.
	Moreover, $(A, \{\,,\,\}, \circ, L_\star, L_\succ)$ is an $(A, [\,,\,], \cdot)$-module Poisson algebra.
\end{proposition}

\begin{remark}
	More generally, post-Poisson algebras admits the following equivalent characterization. 
	Let $A$ be a vector space equipped with binary operations $\{\,,\,\}, \star, \circ, \succ: A \otimes A \to A$.
	Define $[\,,\,]$ and $\cdot$ by Eqs.~\eqref{assoLie} and \eqref{assoCom}, respectively. 
	Then $(A, \{\,,\,\}, \star, \circ, \succ)$ is a post-Poisson algebra if and only if $(A, \{\,,\,\}, \circ)$ and $(A, [\,,\,], \cdot)$ are Poisson algebras, and $(A, \{\,,\,\}, \circ, L_\star, L_\succ)$ is an $(A, [\,,\,], \cdot)$-module Poisson algebra.
\end{remark}

The following result establishes the relationship between $\mathcal{O}$-operators and post-Poisson algebras.
In particular, it shows that an $\mathcal{O}$-operator induces a new Poisson algebra structure and gives a homomorphism between Poisson algebras.
\begin{proposition}[\cite{ni2013poisson}]\label{prop:o2pp}
	Let $(A, [\,,\,], \cdot)$ be a Poisson algebra and $(V, [\,,\,]_V, \cdot_V, \rho, \zeta)$ be an $A$-module Poisson algebra.
	Let $T: V \rightarrow A$ be an $\mathcal{O}$-operator of weight $\lambda$ on $(A, [\,,\,], \cdot)$ associated to $(V, [\,,\,]_V, \cdot_V, \rho, \zeta)$.
	Define the following binary operations on $V$:
	\begin{equation}
		\{u, v\} = \lambda [u, v]_V, \;
		u \star v = \rho(T(u))v, \;
		u \circ v = \lambda u \cdot_V v, \;
		u \succ v = \zeta(T(u))v, \;
		\forall u, v \in V. \label{eq:oo2pp}
	\end{equation}
	Then $(V, \{\,,\,\}, \star, \circ, \succ)$ is a post-Poisson algebra.
	Define binary operations $\{\,,\,\}_T, \circ_T: V \otimes V \rightarrow V$ on $V$ by (for all $u, v \in V$)
	\begin{align}
		\{u, v\}_T & := \rho(T(u))v - \rho(T(v))u + \lambda [u, v]_V, \label{eq:soo2p1}   \\
		u \circ_T v & := \zeta(T(u))v + \zeta(T(v))u + \lambda u \cdot_V v. \label{eq:soo2p2}
	\end{align}
	Then $(V, \{\,,\,\}_T, \circ_T)$ is a Poisson algebra, as the associated Poisson algebra of $(V$, $\{\,,\,\}$, $\star$, $\circ$, $\succ)$.
	Moreover, $T$ is a homomorphism of Poisson algebras between $(V, \{\,,\,\}_T, \circ_T)$ and $(A, [\,,\,], \cdot)$.
\end{proposition}

\begin{corollary}\label{RBtopostPois}
	Let $(A, [\,,\,], \cdot)$ be a Poisson algebra and $T: A \rightarrow A$ be a Rota-Baxter operator of weight $\lambda$ on $(A, [\,,\,], \cdot)$.
	Then there is a post-Poisson algebra structure $(A, \{\,,\,\}, \star, \circ, \succ)$ on $A$ given by
	\begin{align*}
		\{a, b\} = \lambda [a, b], \;
		a \star b = [T(a), b], \;
		a \circ b = \lambda a \cdot b, \;
		a \succ b = T(a) \cdot b, \;
		\forall a, b \in A.
	\end{align*}
	If in addition $T$ is invertible, then there is a post-Poisson algebra structure $(A, \{\,,\,\}_T, \star_T, \circ_T, \succ_T)$ on $A$, whose associated Poisson algebra is $(A, [\,,\,], \cdot)$, given by (for all $a, b \in A$)
	\begin{align*}
		\{a, b\}_T & = \lambda T([T^{-1}(a), T^{-1}(b)]), \;\;\;\;
		a \star_T b = T([a, T^{-1}(b)]), \;              \\
		a \circ_T b & = \lambda T(T^{-1}(a) \cdot T^{-1}(b)), \;\quad
		a \succ_T b = T(a \cdot T^{-1}(b)).
	\end{align*}
\end{corollary}
\begin{proof}
	Since $(A, [\,,\,], \cdot, \ad_{[,]}, L_{\cdot})$ is an $A$-module Poisson algebra, Proposition~\ref{prop:o2pp} implies that $(A$, $\{\,,\,\}$, $\star$, $\circ$, $\succ)$ is a post-Poisson algebra.
	When $T$ is invertible, this structure can be transformed into $(A$, $\{\,,\,\}_T$, $\star_T$, $\circ_T$, $\succ_T)$, whose associated Poisson algebra is precisely $(A, [\,,\,], \cdot)$.
\end{proof}

\begin{example}\label{HesPois}
	Let $(A, [\ ,\ ] , \cdot )$ be a 3-dimensional Poisson algebra with a basis $\{e_1, e_2, e_3\}$ whose non-zero bracket and multiplication are given by
	\begin{equation}
		[e_1,e_2] = e_3 = - [e_2, e_1], \quad e_1 \cdot e_2 = e_3 = e_2 \cdot e_1. \label{HPCF1}
	\end{equation}
	Let $P: A \to A$ be the linear map defined by
	\begin{equation}
		P(e_1) = e_1, \quad P(e_2) = 2e_2, \quad P(e_3) = \frac{1}{2}e_3. \label{HPCF1RB}
	\end{equation}
	Then $P$ is a Rota-Baxter operator of weight $1$ on $(A, [\ ,\ ], \cdot)$.
	By Corollary~\ref{RBtopostPois}, there is a post-Poisson algebra $(A$, $\{\,,\,\}$, $\star$, $\circ$, $\succ)$, whose non-zero binary operations are given by
	\begin{equation*}
		\{ e_1,e_2\} \!=\! e_3 \!=\! -\{e_2, e_1\}, e_1 \star e_2 \!=\! e_3, e_2 \star e_1 \!=\! -2e_3, e_1 \circ e_2 \!=\! e_3 \!=\! e_2 \circ e_1, e_1 \succ e_2 \!=\! e_3, e_2 \succ e_1 \!=\! 2e_3.
	\end{equation*}
\end{example}

\begin{theorem}\label{thm:ippeq}
	Let $(A, [\,,\,], \cdot)$ be a Poisson algebra.
	Then there is a post-Poisson algebra structure on $A$, whose associated Poisson algebra is $(A, [\,,\,], \cdot)$, if and only if there is an invertible $\mathcal{O}$-operator of weight $1$ on $(A, [\,,\,], \cdot)$ associated to an $A$-module Poisson algebra.
\end{theorem}
\begin{proof}
	Suppose that $(A, \{\,,\,\}, \star, \circ, \succ)$ is a post-Poisson algebra, whose associated Poisson algebra is $(A, [\,,\,], \cdot)$.
	Then $\id: A \rightarrow A$ is an invertible $\mathcal{O}$-operator of weight $1$ on $(A, [\,,\,], \cdot)$ associated to the $A$-module Poisson algebra $(A, \{\,,\,\}, \circ, L_\star, L_\succ)$.

	Conversely, suppose that $T: V \rightarrow A$ is an invertible $\mathcal{O}$-operator of weight $1$ on $(A, [\,,\,], \cdot)$ associated to an $A$-module Poisson algebra $(V, [\,,\,]_V, \cdot_V, \rho, \zeta)$.
	Consequently, Proposition~\ref{prop:o2pp} yields a post-Poisson algebra structure $(V, \{\,,\,\}, \star, \circ, \succ)$ via Eq.~\eqref{eq:oo2pp} with $\lambda = 1$.
	Since $T$ is invertible, this structure can be transformed into $(A, [\,,\,]_T, \star_T, \cdot_T, \succ_T)$, explicitly given by (for all $a, b \in A$)
	\begin{align*}
		[a, b]_T  & = T([T^{-1}(a), T^{-1}(b)]_V), \;\;\;\;
		a \star_T b = T(\rho(a)T^{-1}(b)), \;          \\
		a \cdot_T b & = T(T^{-1}(a) \cdot_V T^{-1}(b)), \;\,\quad
		a \succ_T b = T(\zeta(a) T^{-1}(b)),
	\end{align*}
	whose associated Poisson algebra is precisely $(A, [\,,\,], \cdot)$.
\end{proof}

Proposition~\ref{prop:o2pp} demonstrates that an $\mathcal{O}$-operator yields a new Poisson algebra structure on the underlying module space.
At the end of this subsection, we provide a necessary and sufficient condition for a general linear map $T$ to induce such a new Poisson algebra structure.
\begin{lemma}\label{lem:npaeq}
	Let $(A, [\,,\,], \cdot)$ be a Poisson algebra and $(V, [\,,\,]_V, \cdot_V, \rho, \zeta)$ be an $A$-module Poisson algebra.
	Let $T: V \rightarrow A$ be a linear map and $\lambda \in \mathbf{k}$.
	Define binary operations $\{\,,\,\}_T, \circ_T: V \otimes V \rightarrow V$ by Eqs.~\eqref{eq:soo2p1} and \eqref{eq:soo2p2} respectively.
	Then $(V, \{\,,\,\}_T, \circ_T)$ is a Poisson algebra if and only if for all $u, v, w \in V$:
	\begin{align}
		\rho\big(\mathcal{L}_{T}^{\lambda}(u, v)\big)w + \rho\big(\mathcal{L}_{T}^{\lambda}(v, w))\big) u + \rho\big(\mathcal{L}_{T}^{\lambda}(w, u)\big)v  & = 0, \label{eq:soo1} \\
		\zeta\big(\mathcal{A}_{T}^{\lambda}(u, v)\big)w - \zeta\big(\mathcal{A}_{T}^{\lambda}(v, w)\big) u                           & = 0, \label{eq:soo2} \\
		\zeta\big( \mathcal{L}_{T}^{\lambda}(u, v) \big)w + \zeta\big( \mathcal{L}_{T}^{\lambda}(u, w) \big)v + \rho\big(\mathcal{A}_{T}^{\lambda}(v, w)\big)u & = 0, \label{eq:soo3}
	\end{align}
	where
	\begin{align*}
		\mathcal{L}_{T}^{\lambda}(u, v) & = [T(u), T(v)] - T(\rho(T(u))v - \rho(T(v))u + \lambda [u, v]_V),     \\
		\mathcal{A}_{T}^{\lambda}(u, v) & = T(u) \cdot T(v) - T(\zeta(T(u))v + \zeta(T(v))u + \lambda u \cdot_V v).
	\end{align*}
\end{lemma}
\begin{proof}
	For all $u, v, w \in V$, it is easy to check that
	\begin{align*}
		 & \{\{u,v\}_T, w\}_T + \{\{v,w\}_T, u\}_T + \{\{w,u\}_T, v\}_T = - \rho\big(\mathcal{L}_{T}^{\lambda}(u, v)\big)w - \rho\big( \mathcal{L}_{T}^{\lambda}(v, w) \big) u - \rho\big( \mathcal{L}_{T}^{\lambda}(w, u) \big)v,    \\
		 & (u \circ_T v) \circ_T w - u \circ_T (v \circ_T w) = -\zeta\big( \mathcal{A}_{T}^{\lambda}(u, v) \big)w + \zeta\big( \mathcal{A}_{T}^{\lambda}(v, w) \big) u,                                  \\
		 & \{u, v \circ_T w\}_T - \{u, v\}_T \circ_T w - v \circ_T \{u, w\}_T = \zeta\big( \mathcal{L}_{T}^{\lambda}(u, v) \big)w + \zeta\big( \mathcal{L}_{T}^{\lambda}(u, w) \big)v + \rho\big( \mathcal{A}_{T}^{\lambda}(v, w) \big)u.
	\end{align*}
	Hence, $(V, \{\,,\,\}_T, \circ_T)$ is a Poisson algebra if and only if Eqs.~\eqref{eq:soo1}-\eqref{eq:soo3} hold.
\end{proof}

\subsection{Extended \texorpdfstring{$\mathcal{O}$}{O}-operators on Poisson algebras}
In this subsection, we introduce the notion of extended $\mathcal{O}$-operators on Poisson algebras as a generalization of $\mathcal{O}$-operators, following the philosophy of~\cite{bai2010nonabelian, bai2012O}.

\begin{definition}
	Let $(A, [\,,\,], \cdot)$ be a Poisson algebra, $(V, [\,,\,]_V, \cdot_V, \rho, \zeta)$ be an $A$-module Poisson algebra, and $\kappa, \mu \in \mathbf{k}$.
	A linear map $S: V \rightarrow A$ is called
	\begin{enumerate}[label=(\roman*)]
		\item
		   {\bf balanced} associated to $(V, \rho, \zeta)$ if
		   \begin{equation}
			   \rho(S(u))v = -\rho(S(v))u, \;\;
			   \zeta(S(u))v = \zeta(S(v))u, \;\;
			   \forall u, v \in V. \label{eq:bal}
		   \end{equation}

		\item
		   {\bf $A$-invariant of mass $\kappa$} associated to $(V, \rho, \zeta)$ if
		   \begin{equation}
			   \kappa[a, S(v)] = \kappa S(\rho(a)v), \;\;
			   \kappa(a \cdot S(v)) = \kappa S(\zeta(a) v), \;\;
			   \forall a \in A, \; v \in V. \label{eq:ainv}
		   \end{equation}

		\item
		   a {\bf balanced $A$-module homomorphism} associated to $(V, \rho, \zeta)$ if $S: V \rightarrow A$ is balanced and $A$-invariant of nonzero mass associated to $(V, \rho, \zeta)$.
		   That is, $S: V \rightarrow A$ is both an $A$-module homomorphism from $(V, \rho, \zeta)$ to $(A, \ad_{[,]}, L_\cdot)$ and a balanced linear map associated to $(V, \rho, \zeta)$.
		   
		\item
		   {\bf equivalent of mass $\mu$} associated to $(V, [\,,\,]_V, \cdot_V, \rho, \zeta)$ if
				\begin{equation}
					\mu \rho\big( S([u, v]_V) \big) w = \mu [\rho(S(u))v, w]_V, \;
					\mu \zeta(S(u \cdot_V v))w =
					\mu \big(\zeta(S(u))v\big) \cdot_V w, \;\;
					\forall u,v,w \in V. \label{eq:eqvm}
				\end{equation}
		
	\end{enumerate}
\end{definition}

\begin{definition}
	Let $(A, [\,,\,], \cdot)$ be a Poisson algebra, $(V, [\,,\,]_V, \cdot_V, \rho, \zeta)$ be an $A$-module Poisson algebra, and $\kappa, \mu, \lambda \in \mathbf{k}$.
	Let $T, S: V \to A$ be linear maps.
	Then $T$ is called an {\bf extended $\mathcal{O}$-operator of weight $\lambda$ with extension $S$ of mass $(\kappa, \mu)$} on $(A, [\,,\,], \cdot)$ associated to $(V, [\,,\,]_V, \cdot_V, \rho, \zeta)$ if for all $u, v\in V$,
	\begin{align}
		[T(u), T(v)] - T\big(\rho(T(u))v - \rho(T(v))u + \lambda [u,v]_V \big)     & = \kappa [S(u), S(v)] + \mu S([u, v]_V), \label{eq:ext-o-lie}     \\
		T(u) \cdot T(v) - T\big(\zeta(T(u))v + \zeta(T(v))u + \lambda u \cdot_V v\big) & = \kappa S(u) \cdot S(v) + \mu S(u \cdot_V v). \label{eq:ext-o-assoc}
	\end{align}
\end{definition}

\begin{remark}
	The parameters $\kappa, \mu$ and $\lambda$ in the above definitions are introduced so that different cases could be treated uniformly when these parameters vary.
\end{remark}

\begin{example}
	Let $(A, [\,,\,], \cdot)$ be a Poisson algebra.
	\begin{enumerate}[label=(\roman*)]
		\item
		   If $S = 0$, then an extended $\mathcal{O}$-operator of weight $\lambda$ with extension $S$ of mass $(\kappa, \mu)$ reduces to an $\mathcal{O}$-operator of weight $\lambda$ in Definition~\ref{def:oo}.

		\item
		   If $(V, [\,,\,]_V, \cdot_V, \rho, \zeta) = (A, [\,,\,], \cdot, \ad_{[,]}, L_\cdot)$, then $S = \id: A \rightarrow A$ is balanced and $A$-invariant of mass $\kappa$ associated to $(A, \ad_{[,]}, L_\cdot)$, and equivalent of mass $\mu$ associated to $(A, [\,,\,], \cdot, \ad_{[,]}, L_\cdot)$ for all $\kappa, \mu \in \mathbf{k}$.

		\item
		   Every $A$-module homomorphism from $(V, \rho, \zeta)$ to $(A, \ad_{[,]}, L_\cdot)$ is $A$-invariant of mass $\kappa$ associated to $(V, \rho, \zeta)$ for all $\kappa \in \mathbf{k}$.
	\end{enumerate}
\end{example}

\begin{lemma}\label{lem:baie}
	Let $(A, [\,,\,], \cdot)$ be a Poisson algebra and $(V, [\,,\,]_V, \cdot_V, \rho, \zeta)$ be an $A$-module Poisson algebra.
	Suppose that $S: V \rightarrow A$ is balanced associated to $(V, \rho, \zeta)$.
	Then
	\begin{enumerate}[label=(\roman*)]
		\item\label{it:balg1}
		   For all $u, v, w \in V$,
		   \begin{align}
			   \zeta(S([u, v]_V))w + \zeta(S([u, w]_V))v + \rho(S(v \cdot_V w)) u & = 0. \label{eq:ipf1}
		   \end{align}

		\item\label{it:balg2}
		   If $S$ in addition is $A$-invariant of mass $\kappa$ associated to $(V, \rho, \zeta)$,
		   then
		   \begin{align}
			   \kappa \rho(S(u)) \rho(S(v)) w + \kappa \rho(S(v)) \rho(S(w)) u + \kappa \rho(S(w)) \rho(S(u)) v & = 0, \label{eq:ipf2} \\
			   \kappa \zeta(S(u))\zeta(S(v)) w - \kappa \zeta(S(w)) \zeta(S(u)) v                & = 0, \label{eq:ipf3} \\
			   \kappa \zeta([S(u), S(v)])w + \kappa \zeta([S(u), S(w)])v + \kappa \rho(S(v) \cdot S(w)) u    & = 0. \label{eq:ipf4}
		   \end{align}

		\item\label{it:balg3}
		   If $S$ in addition is equivalent of mass $\mu$ associated to $(V, [\,,\,]_V, \cdot_V, \rho, \zeta)$,
		   then
		   \begin{align}
			   \mu \rho(S([u, v]_V))w + \mu \rho(S([v, w]_V))u + \mu \rho(S([w, u]_V))v & = 0, \label{eq:ipf5} \\
			   \mu \zeta(S(u \cdot_V v)) w - \mu \zeta(S(v \cdot_V w)) u        & = 0. \label{eq:ipf6}
		   \end{align}
	\end{enumerate}
\end{lemma}
\begin{proof}
	\ref{it:balg1}. For all $u, v, w \in V$, we have
	\begin{align*}
		 & \zeta(S([u, v]_V))w + \zeta(S([u, w]_V))v + \rho(S(v \cdot_V w)) u                      \\
		 & \overset{\eqref{eq:bal}}{=} \zeta(S(w))[u, v]_V + \zeta(S(v))[u, w]_V - \rho(S(u))(v \cdot_V w)        \\
		 & \overset{\eqref{eq:pmd}}{=} \zeta(S(w))[u, v]_V + \zeta(S(v))[u, w]_V + [v, \zeta(S(u))w] + [w, \zeta(S(u))v] \\
		 & \overset{\eqref{eq:bal}}{=} \zeta(S(w))[u, v]_V + \zeta(S(v))[u, w]_V + [v, \zeta(S(w))u] + [w, \zeta(S(v))u] \\
		 & \overset{\eqref{eq:pmd}}{=} -(\rho(S(w))v) \cdot u -(\rho(S(v))w) \cdot u \overset{\eqref{eq:bal}}{=} 0.
	\end{align*}

	\ref{it:balg2}. For all $u, v, w \in V$, we have
	\begin{equation*}
		\kappa \rho([S(u), S(v)])w \overset{\eqref{eq:ainv}}{=} \kappa \rho\big( S(\rho(S(u))v) \big)w \overset{\eqref{eq:bal}}{=} -\kappa \rho(S(w)) \rho(S(u))v
	\end{equation*}
	On the other hand,
	\begin{equation*}
		\kappa \rho([S(u), S(v)])w \!=\! \kappa \big(\rho(S(u)) \rho(S(v)) - \rho(S(v)) \rho(S(u)) \big) w \!\overset{\eqref{eq:bal}}{=}\! \kappa \rho(S(u)) \rho(S(v))w + \kappa \rho(S(v)) \rho(S(w)) u.
	\end{equation*}
	Therefore, Eq.~\eqref{eq:ipf2} holds.
	By
	\begin{equation*}
		\kappa \zeta(S(u))\zeta(S(v)) w = \kappa \zeta(S(u) \cdot S(v))w \overset{\eqref{eq:ainv}}{=} \kappa \zeta\big( S(\zeta(S(u))v)\big)w \overset{\eqref{eq:bal}}{=} \kappa \zeta(S(w))\zeta(S(u))v,
	\end{equation*}
	we show that Eq.~\eqref{eq:ipf3} holds.
	Finally, we have
	\begin{align*}
		\kappa \zeta([S(u), S(v)])w \overset{\eqref{eq:ainv}}{=} \kappa \zeta\big( S(\rho(S(u))v)\big)w \overset{\eqref{eq:bal}}{=} \kappa \zeta(S(w)) \rho(S(u))v,
	\end{align*}
	and
	\begin{align*}
		\kappa \rho(S(u) \cdot S(v)) w \overset{\eqref{eq:ainv}}{=} \kappa \rho\big( S(\zeta(S(u))v)\big)w \overset{\eqref{eq:bal}}{=} -\kappa \rho(S(w)) \zeta(S(u))v.
	\end{align*}
	Thus,
	\begin{align*}
		 & \kappa \zeta([S(u), S(v)])w + \kappa \zeta([S(u), S(w)])v + \kappa \rho(S(v) \cdot S(w)) u                 \\
		 & = \kappa \zeta(S(w)) \rho(S(u))v + \kappa \zeta([S(u), S(w)])v - \kappa \rho(S(u)) \zeta(S(v))w              \\
		 & \overset{\eqref{eq:ainv}}{=} \kappa \zeta(S(w)) \rho(S(u))v + \kappa \zeta([S(u), S(w)])v - \kappa \rho(S(u)) \zeta(S(w))v \\
		 & \overset{\eqref{eq:repp}}{=} 0.
	\end{align*}
	That is, Eq.~\eqref{eq:ipf4} holds.

	\ref{it:balg3}. For all $u, v, w \in V$, we have
	\begin{align*}
		 & \mu \rho(S([u, v]_V)) w \overset{\eqref{eq:eqvm}}{=} \mu [\rho(S(u))v, w] \overset{\eqref{eq:pmc}}{=} \mu \rho(S(u))[v, w]_V - \mu [v, \rho(S(u))w]_V     \\
		 & \overset{\eqref{eq:eqvm}}{=} \mu \rho(S(u))[v, w]_V + \mu \rho(S([u, w]_V)) v \overset{\eqref{eq:bal}}{=} -\mu \rho(S([v, w]_V)) u - \mu \rho(S([w, u]_V)) v.
	\end{align*}
	Therefore, Eq.~\eqref{eq:ipf5} holds.
	\begin{align*}
		\mu \zeta(S(u \cdot_V v)) w \overset{\eqref{eq:eqvm}}{=} \mu (\zeta(S(u))v) \cdot_V w \overset{\eqref{eq:pmc}}{=} \mu \zeta(S(u)) (v \cdot_V w) \overset{\eqref{eq:bal}}{=} \mu \zeta(S(v \cdot_V w)) u.
	\end{align*}
	That is, Eq.~\eqref{eq:ipf6} holds.
\end{proof}

\begin{proposition}\label{prop:bae2amp}
	Let $(A, [\,,\,], \cdot)$ be a Poisson algebra and $(V, [\,,\,]_V, \cdot_V, \rho, \zeta)$ be an $A$-module Poisson algebra.
	Suppose that $S: V \to A$ is a balanced $A$-module homomorphism associated to $(V, \rho, \zeta)$ satisfying Eq.~\eqref{eq:eqvm} with $\mu = \lambda$.
	Then $(V, \{\,,\,\}_{+}, \circ_{+}, \rho, \zeta)$ (resp. $(V, \{\,,\,\}_{-}, \circ_{-}, \rho, \zeta)$) is an $A$-module Poisson algebras, where $\{\,,\,\}_{+}$ and $\circ_{+}$ (resp. $\{\,,\,\}_{-}$ and $\circ_{-}$) are defined by (for all $u, v \in V$)
	\begin{align*}
		\{u, v\}_{+} = \lambda [u, v]_V - 2\rho(S(u))v \quad \text{(resp. $\{u, v\}_{-} = \lambda [u, v]_V + 2\rho(S(u))v$)}, \\
		u \circ_{+} v = \lambda u \cdot_V v - 2 \zeta(S(u))v \quad \text{(resp. $u \circ_{-} v = \lambda u \cdot_V v + 2 \zeta(S(u))v$)}.
	\end{align*}
\end{proposition}
\begin{proof}
	For brevity, we prove only the case $(V, \{\,,\,\}_{+}, \circ_{+}, \rho, \zeta)$; the argument for $(V, \{\,,\,\}_{-}, \circ_{-}, \rho, \zeta)$ proceeds analogously.
	Let $u, v, w \in V$. Then
	\begin{align*}
		& \{\{u, v\}_{+}, w\}_{+} = \lambda^2 [[u, v]_V, w]_V - 2 \lambda [\rho(S(u))v, w] + 2 \lambda \rho(S(w))([u, v]_V) - 4 \rho(S(w))\rho(S(u))v \\
		& \overset{\eqref{eq:bal},\eqref{eq:eqvm}}{=} \lambda^2 [[u, v]_V, w]_V - 4 \lambda \rho(S([u, v]_V))w - 4 \rho(S(w))\rho(S(u))v
	\end{align*}
	Therefore,
	\begin{equation*}
		\{\{u, v\}_{+}, w\}_{+} + \{\{v, w\}_{+}, u\}_{+} + \{\{w, u\}_{+}, v\}_{+} \overset{\eqref{eq:ipf2},\eqref{eq:ipf5}}{=} 0.
	\end{equation*}
	That is, $(V, \{\,,\,\}_{+})$ is a Lie algebra.
	By
	\begin{align*}
		& (u \circ_{+} v) \circ_{+} w = \lambda^2 (u \cdot_V v) \cdot_V w - 2\lambda (\zeta(S(u))v) \cdot_V w - 2\lambda \zeta(S(u \cdot_V v))w + 4 \zeta(S(\zeta(S(u))v))w \\
		& \overset{\eqref{eq:bal}, \eqref{eq:eqvm}}{=} \lambda^2 (u \cdot_V v) \cdot_V w - 4 \lambda \zeta(S(u \cdot_V v))w + 4 \zeta(S(w))\zeta(S(u))v,
	\end{align*}
	we have
	\begin{equation*}
		(u \circ_{+} v) \circ_{+} w - u \circ_{+} (v \circ_{+} w) \overset{\eqref{eq:ipf3}, \eqref{eq:ipf6}}{=} 0.
	\end{equation*}
	That is, $(V, \circ_{+})$ is a commutative associative algebra.
	Furthermore, we have
	\begin{align*}
		& \{u, v \circ_{+} w\}_{+} - \{u, v\}_{+} \circ_{+} w - v \circ_{+} \{u, w\}_{+}                          \\
		& =\lambda [u, \lambda v \cdot_V w - 2 \zeta(S(v))w]_V - 2 \rho(S(u)) (\lambda v \cdot_V w - 2 \zeta(S(v))w)            \\
		& \quad - \lambda (\lambda [u, v]_V - 2 \rho(S(u))v) \cdot_V w + 2 \zeta(S(w))(\lambda [u, v]_V - 2 \rho(S(u))v)          \\
		& \quad - \lambda v \cdot_V(\lambda [u, w]_V - 2 \rho(S(u))w) + 2 \zeta(S(v))(\lambda [u, w]_V - 2 \rho(S(u))w)          \\
		& \overset{\eqref{eq:pmd}}{=} - 2 \lambda [u, \zeta(S(v))w]_V+ 2 \lambda \zeta(S(w))([u, v]_V) + 2 \lambda \zeta(S(v))([u, w]_V)  \\
		& \quad + 4\rho(S(u))\zeta(S(v))w - 4\zeta(S(w))\rho(S(u))v - 4\zeta(S(v))\rho(S(u))w                       \\
		& \overset{\eqref{eq:bal}}{=} - 2 \lambda [u, \zeta(S(v))w]_V + 2 \lambda \zeta(S(w))([u, v]_V) + 2 \lambda \zeta(S(v))([u, w]_V) \\
		& \quad\;\; - 4\rho(S(v) \cdot S(w))u - 4\zeta([S(u), S(v)])w - 4\zeta([S(u), S(w)])v                         \\
		& \overset{\eqref{eq:ipf4}}{=} - 2 \lambda [u, \zeta(S(v))w]_V+ 2 \lambda \zeta(S(w))([u, v]_V) + 2 \lambda \zeta(S(v))([u, w]_V).
	\end{align*}
	Note that
	\begin{align*}
		& - [u, \zeta(S(v))w]_V+ \zeta(S(w))([u, v]_V) + \zeta(S(v))([u, w]_V)                                                     \\
		& \overset{\eqref{eq:pmd}}{=} [u, \zeta(S(v))w]_V + (\rho(S(w))u) \cdot_V v + (\rho(S(v))u) \cdot_V w \overset{\eqref{eq:pmd}}{=} [u, \zeta(S(v))w]_V - \rho(S(u))(v \cdot_V w)
	\end{align*}
	and
	\begin{align*}
		& - [u, \zeta(S(v))w]_V + \zeta(S(w))([u, v]_V) + \zeta(S(v))([u, w]_V)                             \\
		& \overset{\eqref{eq:ipf1}}{=} - [u, \zeta(S(v))w]_V - \rho(S(v \cdot_V w)) u = - [u, \zeta(S(v))w]_V + \rho(S(u))(v \cdot_V w).
	\end{align*}
	Hence,
	\begin{align*}
		& \{u, v \circ_{+} w\}_{+} - \{u, v\}_{+} \circ_{+} w - v \circ_{+} \{u, w\}_{+}              \\
		& = - 2 \lambda [u, \zeta(S(v))w]_V+ 2 \lambda \zeta(S(w))([u, v]_V) + 2 \lambda \zeta(S(v))([u, w]_V) = 0.
	\end{align*}
	Therefore, $(V, \{\,,\,\}_{+}, \circ_{+})$ is a Poisson algebra.
	One can now readily verify that $(V, \{\,,\,\}_{+}, \circ_{+}, \rho, \zeta)$ is an $A$-module Poisson algebra.
\end{proof}

\begin{theorem}\label{thm:eo2ps}
	Let $(A, [\,,\,], \cdot)$ be a Poisson algebra, $(V, [\,,\,]_V, \cdot_V, \rho, \zeta)$ be an $A$-module Poisson algebra, and $T, S: V \rightarrow A$ be two linear maps.
	If $S$ satisfies Eqs.~\eqref{eq:bal}, \eqref{eq:ainv} and \eqref{eq:eqvm}, and $T$ is an extended $\mathcal{O}$-operator of weight $\lambda$ with extension $S$ of mass $(\kappa, \mu)$ on $(A, [\,,\,], \cdot)$ associated to $(V, [\,,\,]_V, \cdot_V, \rho, \zeta)$,
	then $(V$, $\{\,,\,\}_T$, $\circ_T)$ is a Poisson algebra, where $\{\,,\,\}_T$ and $\circ_T$ are defined by Eqs.~\eqref{eq:soo2p1} and \eqref{eq:soo2p2} respectively.
\end{theorem}
\begin{proof}
	By Lemma~\ref{lem:npaeq}, it suffices to show that Eqs.~\eqref{eq:soo1}, \eqref{eq:soo2} and \eqref{eq:soo3} hold.
	Let $u, v, w \in V$.
	Since $T$ is an extended $\mathcal{O}$-operator of weight $\lambda$ with extension $S$ of mass $(\kappa, \mu)$ on $(A, [\,,\,], \cdot)$ associated $(V, [\,,\,]_V, \cdot_V, \rho, \zeta)$, we have
	\begin{equation*}
		\mathcal{L}_{T}^{\lambda}(u, v) = \kappa [S(u), S(v)] + \mu S([u, v]_V), \;\;
		\mathcal{A}_{T}^{\lambda}(u, v) = \kappa S(u) \cdot S(v) + \mu S(u \cdot_V v).
	\end{equation*}
	Therefore, by Lemma~\ref{lem:baie}, we have
	\begin{align*}
		 & \rho\big(\mathcal{L}_{T}^{\lambda}(u, v)\big)w + \rho\big(\mathcal{L}_{T}^{\lambda}(v, w)\big) u + \rho\big(\mathcal{L}_{T}^{\lambda}(w, u))\big)v                \\
		 & = \kappa \rho\big([S(u), S(v)]\big)w + \mu \rho\big(S([u, v]_V)\big)w + \kappa \rho\big([S(v), S(w)]\big)u + \mu \rho\big(S([v, w]_V)\big)u                    \\
		 & \quad + \kappa \rho\big([S(w), S(u)]\big)v + \mu \rho\big(S([w, u]_V)\big)v \overset{\eqref{eq:ipf2} ,\eqref{eq:ipf5}}{=} 0,                            \\
		 & \zeta\big(\mathcal{A}_{T}^{\lambda}(u, v)\big)w -\zeta\big(\mathcal{A}_{T}^{\lambda}(v, w)\big) u                                         \\
		 & = \kappa \zeta(S(u) \cdot S(v) )w + \mu \zeta(S(u \cdot_V v)) w - \kappa \zeta(S(v) \cdot S(w) )u - \mu \zeta(S(v \cdot_V w)) u \overset{\eqref{eq:ipf3}, \eqref{eq:ipf6}}{=} 0, \\
		 & \zeta\big( \mathcal{L}_{T}^{\lambda}(u, v) \big)w + \zeta\big( \mathcal{L}_{T}^{\lambda}(u, w) \big)v + \rho\big(\mathcal{A}_{T}^{\lambda}(v, w)\big)u               \\
		 & = \kappa \zeta\big([S(u), S(v)]\big)w + \mu \zeta\big(S([u, v]_V)\big)w + \kappa \zeta\big([S(u), S(w)]\big)v + \mu \zeta\big(S([u, w]_V)\big)v                  \\
		 & \quad + \kappa \rho(S(v) \cdot S(w) )u + \mu \rho(S(v \cdot_V w)) u \overset{\eqref{eq:ipf1}, \eqref{eq:ipf4}}{=} 0.
	\end{align*}
	That is, Eqs.~\eqref{eq:soo1}-\eqref{eq:soo3} hold.
	The proof is completed.
\end{proof}

\begin{corollary}
	Let $(A, [\,,\,], \cdot)$ be a Poisson algebra, and $S, T: A \rightarrow A$ be two linear maps.
	If $S$ is balanced and $A$-invariant of mass $\kappa$ associated to $(A, \ad_{[,]}, L_\cdot)$, and $T: A \rightarrow A$ is an extended $\mathcal{O}$-operator of weight $\lambda$ with extension $S$ of mass $(\kappa, \kappa)$ or $(\kappa, 0)$ associated to $(A, [\,,\,], \cdot, \ad_{[,]}, L_\cdot)$,
	then $(A, \{\,,\,\}_T, \circ_T)$ is a Poisson algebra, where
	$\{\,,\,\}_T, \circ_T: A \otimes A \rightarrow A$ are defined by (for all $a, b \in A$)
	\begin{align*}
		\{a, b\}_T & := [T(a), b] + [a, T(b)] + \lambda [a, b],     \\
		a \circ_T b & := T(a) \cdot b + a \cdot T(b) + \lambda a \cdot b.
	\end{align*}
\end{corollary}
\begin{proof}
	Note that if $S$ is balanced and $A$-invariant of mass $\kappa$ associated to $(A, \ad_{[,]}, L_\cdot)$, then $S$ is equivalent of mass $\kappa$ or $0$ associated to $(A, [\,,\,], \cdot, \ad_{[,]}, L_\cdot)$.
	Therefore, the conclusion follows from Theorem~\ref{thm:eo2ps}.
\end{proof}

Let $A$ and $V$ be vector spaces, and $\pi_{\pm}: V \rightarrow A$ be two linear maps.
Set
\begin{equation*}
	T := \frac{1}{2}(\pi_{+} + \pi_{-}), \;\;
	S := \frac{1}{2}(\pi_{+} - \pi_{-}),
\end{equation*}
called the {\bf symmetrizer} and {\bf antisymmetrizer} of $\pi_{\pm}$, respectively.
Note that $\pi_{\pm}$ can be resolved from $T$ and $S$ by $\pi_{\pm} = T \pm S$.

\begin{proposition}\label{prop:pm2p}
	Let $(A, [\,,\,], \cdot)$ be a Poisson algebra, $(V, [\,,\,]_V, \cdot_V, \rho, \zeta)$ be an $A$-module Poisson algebra, and $\pi_{\pm}: V \rightarrow A$ be two linear maps.
	Let $T, S$ be symmetrizer and antisymmetrizer of $\pi_{\pm}$ respectively.
	If $S$ is balanced associated to $(V, \rho, \zeta)$, then $(V, \{\,,\,\}, \circ)$ is a Poisson algebra, where $\{\,,\,\}, \circ: V \otimes V \rightarrow V$ are respectively defined by (for all $u, v \in V$)
	\begin{align*}
		\{u, v\} & := \rho(\pi_{+}(u))v - \rho(\pi_{-}(v))u + \lambda [u, v]_V,   \\
		u \circ v & := \zeta(\pi_{+}(u))v + \zeta(\pi_{-}(v))u + \lambda u \cdot_V v,
	\end{align*}
	if and only if Eqs.~\eqref{eq:soo1}-\eqref{eq:soo3} hold.
	If in addition $T$ is an $\mathcal{O}$-operator of weight $\lambda$ on $(A, [\,,\,], \cdot)$ associated to $(V, [\,,\,]_V, \cdot_V, \rho, \zeta)$, then $(V, \{\,,\,\}, \circ)$ is a Poisson algebra and $T$ is a homomorphism of Poisson algebras between $(V, \{\,,\,\}, \circ)$ and $(A, [\,,\,], \cdot)$.
\end{proposition}
\begin{proof}
	Since $S$ is balanced associated to $(V, \rho, \zeta)$, we have
	\begin{align*}
		\{u, v\} & = \rho(\pi_{+}(u))v - \rho(\pi_{-}(v))u + \lambda [u, v]_V \overset{\eqref{eq:bal}}{=} \rho(T(u))v - \rho(T(v))u + \lambda [u, v]_V,      \\
		u \circ v & = \zeta(\pi_{+}(u))v + \zeta(\pi_{-}(v))u + \lambda u \cdot_V v \overset{\eqref{eq:bal}}{=} \zeta(T(u))v + \zeta(T(v))u + \lambda u \cdot_V v,
	\end{align*}
	for all $u, v \in V$.
	Hence, by Lemma~\ref{lem:npaeq}, $(V, \{\,,\,\}, \circ)$ is a Poisson algebra if and only if Eqs.~\eqref{eq:soo1}-\eqref{eq:soo3} hold.
	Furthermore, if in addition $T$ is an $\mathcal{O}$-operator of weight $\lambda$ on $(A, [\,,\,], \cdot)$ associated to $(V, [\,,\,]_V, \cdot_V, \rho, \zeta)$, then $(V, \{\,,\,\}, \circ)$ is a Poisson algebra and $T$ is a homomorphism of Poisson algebras between $(V, \{\,,\,\}, \circ)$ and $(A, [\,,\,], \cdot)$ by Proposition~\ref{prop:o2pp}.
\end{proof}

The following results provide equivalent characterizations of extended $\mathcal{O}$-operators under certain conditions.

\begin{proposition}\label{prop:iphr}
	Let $(A, [\,,\,], \cdot)$ be a Poisson algebra, $(V, [\,,\,]_V, \cdot_V, \rho, \zeta)$ be an $A$-module Poisson algebra, and $\pi_{\pm}: V \rightarrow A$ be two linear maps.
	Let $T, S$ be symmetrizer and antisymmetrizer of $\pi_{\pm}$ respectively.
	Suppose that $S$ is an $A$-module homomorphism from $(V, \rho, \zeta)$ to $(A, \ad_{[,]}, L_\cdot)$.
	Then $T$ is an extended $\mathcal{O}$-operator of weight $\lambda$ with extension $S$ of mass $(-1, \lambda)$ (resp. $(-1, -\lambda)$) on $(A, [\,,\,], \cdot)$ associated $(V, [\,,\,]_V, \cdot_V, \rho, \zeta)$ if and only if
	\begin{align}
		[\pi_{+}(u), \pi_{+}(v)]  & = \pi_{+}(\{u, v\}_T) \;\; \text{(resp. $[\pi_{-}(u), \pi_{-}(v)] = \pi_{-}(\{u, v\}_T)$)}, \label{eq:pihl} \\
		\pi_{+}(u) \cdot \pi_{+}(v) & = \pi_{+}(u \circ_T v) \;\; \text{(resp. $\pi_{-}(u) \cdot \pi_{-}(v) = \pi_{-}(u \circ_T v)$)}, \label{eq:piha}
	\end{align}
	where $\{\,,\,\}_T$ and $\circ_T$ are defined by Eqs.~\eqref{eq:soo2p1} and \eqref{eq:soo2p2} respectively.
\end{proposition}
\begin{proof}
	We prove only the mass $(-1, \lambda)$ case, and the mass $(-1, -\lambda)$ case is analogous.
	\begin{align*}
		 & [\pi_{+}(u), \pi_{+}(v)] - \pi_{+}(\{u, v\}_T) = [(T+S)(u), (T+S)(v)] - (T+S)(\rho(T(u))v - \rho(T(v))u + \lambda [u,v]_V) \\
		 & = [T(u), T(v)] + [S(u), S(v)] - T(\rho(T(u))v - \rho(T(v))u + \lambda [u,v]_V)                       \\
		 & \quad + [T(u), S(v)] + [S(u), T(v)] - S(\rho(T(u))v - \rho(T(v))u + \lambda [u,v]_V)                    \\
		 & =[T(u), T(v)] + [S(u), S(v)] - T(\rho(T(u))v - \rho(T(v))u + \lambda [u,v]_V)                       \\
		 & \quad + S(\rho(T(u))v) - S(\rho(T(v))u) - S(\rho(T(u))v - \rho(T(v))u + \lambda [u,v]_V)                 \\
		 & =[T(u), T(v)] + [S(u), S(v)] - T(\rho(T(u))v - \rho(T(v))u + \lambda [u,v]_V) - \lambda S([u, v]_V).
	\end{align*}
	That is, Eq.~\eqref{eq:pihl} holds if and only if Eq.~\eqref{eq:ext-o-lie} holds with $\kappa = -1$ and $\mu = \lambda$.
	Similarly, we show that Eq.~\eqref{eq:piha} holds if and only if Eq.~\eqref{eq:ext-o-assoc} holds with $\kappa = -1$ and $\mu = \lambda$.
	Hence, $T$ is an extended $\mathcal{O}$-operator of weight $\lambda$ with extension $S$ of mass $(-1, \lambda)$ on $(A, [\,,\,], \cdot)$ associated $(V, [\,,\,]_V, \cdot_V, \rho, \zeta)$ if and only if Eqs.~\eqref{eq:pihl} and \eqref{eq:piha} hold.
\end{proof}

\begin{remark}
	$(A, \{\,,\,\}_T, \circ_{T})$ may not be a Poisson algebra in Proposition~\ref{prop:iphr}.
	However, by Theorem~\ref{thm:eo2ps}, if Eq.~\eqref{eq:eqvm} holds with $\mu = \lambda$,
	then $(A, \{\,,\,\}_T, \circ_{T})$ is indeed a Poisson algebra under the conditions of Proposition~\ref{prop:iphr}.
\end{remark}

\begin{theorem}\label{thm:opm}
	Let $(A, [\,,\,], \cdot)$ be a Poisson algebra, $(V, [\,,\,]_V, \cdot_V, \rho, \zeta)$ be an $A$-module Poisson algebra, and $\pi_{\pm}: V \rightarrow A$ be two linear maps. 
	Let $T, S$ be symmetrizer and antisymmetrizer of $\pi_{\pm}$ respectively.
	Suppose that $S$ is a balanced $A$-module homomorphism associated to $(V, \rho, \zeta)$ satisfying Eq.~\eqref{eq:eqvm} with $\mu = \lambda$.
	Then the following conditions are equivalent.
	\begin{enumerate}[label=(\roman*)]
		\item\label{it:oeoheo} 
			$T$ is an extended $\mathcal{O}$-operator of weight $\lambda$ with extension $S$ of mass $(-1, \lambda)$ (resp. $(-1, -\lambda)$) on $(A, [\,,\,], \cdot)$ associated to $(V, [\,,\,]_V, \cdot_V, \rho, \zeta)$.
		\item\label{it:oeohh} 
			$(V, \{\,,\,\}_{T}, \circ_{T})$ is a Poisson algebra, and $\pi_{+}$ (resp. $\pi_{-}$) is a homomorphism of Poisson algebras between $(V, \{\,,\,\}_{T}, \circ_{T})$ and $(A, [\,,\,], \cdot)$,
			where $\{\,,\,\}_T$ and $\circ_T$ are defined by Eqs.~\eqref{eq:soo2p1} and \eqref{eq:soo2p2} respectively.
			
		\item\label{it:oeoho}  
			$\pi_{+}$ (resp. $\pi_{-}$) is an $\mathcal{O}$-operator of weight 1 on $(A, [\,,\,], \cdot)$ associated to $(V, \{\,,\,\}_{+}, \circ_{+}, \rho, \zeta)$ (resp. $(V, \{\,,\,\}_{-}, \circ_{-}, \rho, \zeta)$)
			where $(V, \{\,,\,\}_{+}, \circ_{+}, \rho, \zeta)$ (resp. $(V, \{\,,\,\}_{-}, \circ_{-}, \rho, \zeta)$) is the $A$-module Poisson algebra obtained in Proposition~\ref{prop:bae2amp}.
	\end{enumerate}
\end{theorem}
\begin{proof}
	We prove only the $(-1, \lambda)$ case; the $(-1, -\lambda)$ case is analogous.
	
	\ref{it:oeoheo} $\Longrightarrow$ \ref{it:oeohh}.
	By Theorem~\ref{thm:eo2ps}, $(V, \{\,,\,\}_{T}, \circ_{T})$ is a Poisson algebra. 
	Then Proposition~\ref{prop:iphr} shows that $\pi_{+}$ is a homomorphism of Poisson algebras.
	
	\ref{it:oeohh} $\Longrightarrow$ \ref{it:oeoheo}. 
	It follows from Proposition~\ref{prop:iphr}.
	
	\ref{it:oeohh} $\Longrightarrow$ \ref{it:oeoho}.
	For all $u, v \in V$, we have
	\begin{align*}
		&\rho(\pi_{+}(u))v - \rho(\pi_{+}(v))u + \{u,v\}_{+} =\rho(\pi_{+}(u))v - \rho(\pi_{+}(v))u + \lambda [u,v]_V - 2 \rho(S(u))v \\
		&=\rho(T(u))v - \rho(T(v))u + \lambda [u,v]_V + \rho(S(u))v - \rho(S(v))u - 2 \rho(S(u))v \\
		&=\rho(T(u))v - \rho(T(v))u + \lambda [u,v]_V = \{u, v\}_T.
	\end{align*}
	Therefore, 
	\begin{equation*}
		[\pi_{+}(u), \pi_{+}(v)] - \pi_{+}\big(\rho(\pi_{+}(u))v - \rho(\pi_{+}(v))u + \{u,v\}_{+} \big) = [\pi_{+}(u), \pi_{+}(v)] - \pi_{+}(\{u, v\}_T) = 0.
	\end{equation*}
	Similarly, we have
	\begin{equation*}
		\pi_{+}(u) \cdot \pi_{+}(v) - \pi_{+}\big(\zeta(\pi_{+}(u))v + \zeta(\pi_{+}(v))u + u \circ_{+} v \big) = \pi_{+}(u) \cdot \pi_{+}(v) - \pi_{+}(u \cdot_T v) = 0.
	\end{equation*}
	Hence, $\pi_{+}$ is an $\mathcal{O}$-operator of weight 1 on $(A, [\,,\,], \cdot)$ associated to $(V, \{\,,\,\}_{+}, \circ_{+}, \rho, \zeta)$.
	
	\ref{it:oeoho} $\Longrightarrow$ \ref{it:oeohh}.
	By Proposition~\ref{prop:pm2p}, $\pi_{+}$ is a homomorphism of Poisson algebras between $(V, \{\,,\,\}, \circ)$ and $(A, [\,,\,], \cdot)$, where $\{\,,\,\}, \circ: V \otimes V \to V$ is defined by
	\begin{align*}
		\{u, v\} & := \rho(\pi_{+}(u))v - \rho(\pi_{-}(v))u + \{u, v\}_{+},   \\
		u \circ v & := \zeta(\pi_{+}(u))v + \zeta(\pi_{-}(v))u + u \circ_{+} v.
	\end{align*}
	On the other hand, following the proof of the implication ``(\ref{it:oeohh}) $\Longrightarrow$ (\ref{it:oeoho})'', we obtain 
	\begin{align*}
		\{u, v\} = \{u, v\}_T, \;\;
		u \circ v = u \circ_T v, \;\;
		\forall u, v \in V.
	\end{align*}
	Hence, $\pi_{+}$ is a homomorphism of Poisson algebras between $(V, \{\,,\,\}_{T}, \circ_{T})$ and $(A, [\,,\,], \cdot)$.
\end{proof}

\begin{corollary}\label{coro:bah2e}
	Let $(A, [\,,\,], \cdot)$ be a Poisson algebra, $(V, \rho, \zeta)$ be a module of $(A, [\,,\,], \cdot)$, and $S: V \rightarrow A$ be a balanced $A$-module homomorphism associated to $(V, \rho, \zeta)$.
	Then $(V, \{\,,\,\}_{+}, \circ_{+}, \rho, \zeta)$ (resp. $(V, \{\,,\,\}_{-}, \circ_{-}, \rho, \zeta)$) is an $A$-module Poisson algebras, where $\{\,,\,\}_{+}$ and $\circ_{+}$ (resp. $\{\,,\,\}_{-}$ and $\circ_{-}$) are defined by (for all $u, v \in V$)
	\begin{align*}
		\{u, v\}_{+} = - 2\rho(S(u))v \quad \text{(resp. $\{u, v\}_{-} = 2\rho(S(u))v$)}, \\
		u \circ_{+} v = - 2 \zeta(S(u))v \quad \text{(resp. $u \circ_{-} v = 2 \zeta(S(u))v$)}.
	\end{align*}
	Furthermore, the following conditions are equivalent.
	\begin{enumerate}[label=(\roman*)]
		\item 
			$T: V \rightarrow A$ is an extended $\mathcal{O}$-operator of weight $0$ with extension $S$ of mass $(-1, 0)$ on $(A, [\,,\,], \cdot)$ associated to $(V, \rho, \zeta)$.
		
		\item 
			$T+S$ (resp. $T-S$) is a homomorphism of Poisson algebras between $(V, \{\,,\,\}_{T}, \circ_{T})$ and $(A, [\,,\,], \cdot)$,
			where $\{\,,\,\}_T$ and $\circ_T$ are defined by 
			\begin{align*}
				\{u, v\}_T & := \rho(T(u))v - \rho(T(v))u ,   \\
				u \circ_T v & := \zeta(T(u))v + \zeta(T(v))u. 
			\end{align*}
			
		\item 
			$T+S$ (resp. $T-S$) is an $\mathcal{O}$-operator of weight 1 on $(A, [\,,\,], \cdot)$ associated to $(V$, $\{\,,\,\}_{+}$, $\circ_{+}$, $\rho$, $\zeta)$ (resp. $(V$, $\{\,,\,\}_{-}$, $\circ_{-}$, $\rho$, $\zeta)$).
	\end{enumerate}
\end{corollary}
\begin{proof}
	Regarding $(V, \rho, \zeta)$ as an $A$-module Poisson algebra with trivial Poisson algebra structure, the conclusions follow from Proposition~\ref{prop:bae2amp} and Theorem~\ref{thm:opm}.
\end{proof}

\begin{corollary}\label{coro:eoqadj}
	Let $(A, [\,,\,], \cdot)$ be a Poisson algebra and $S: A \rightarrow A$ be a balanced $A$-module homomorphism associated to $(A, \ad_{[,]}, L_\cdot)$.
	Then $(A, \{\,,\,\}_{+}, \circ_{+}, \ad_{[,]}, L_\cdot)$ (resp. $(A, \{\,,\,\}_{-}, \circ_{-}, \ad_{[,]}, L_\cdot)$) is an $A$-module Poisson algebra where
	\begin{align*}
		\{a, b\}_{+} & = \lambda [a, b] - 2[S(a), b] \quad \text{(resp. $\{a, b\}_{-} = \lambda [a, b] + 2[S(a), b]$)},        \\
		a \circ_{+} b & = \lambda a \cdot b - 2 S(a) \cdot b \quad \text{(resp. $a \circ_{-} b = \lambda a \cdot b + 2 S(a) \cdot b$)}.
	\end{align*}
	Furthermore, the following conditions are equivalent.
	\begin{enumerate}[label=(\roman*)]
		\item 
			$T$ is an extended $\mathcal{O}$-operator of weight $\lambda$ with extension $S$ of mass $(-1, \lambda)$ (resp. $(-1, -\lambda)$) on $(A, [\,,\,], \cdot)$ associated to $(A, [\,,\,], \cdot, \ad_{[,]}, L_\cdot)$.
		
		\item 
			$T+S$ (resp. $T-S$) is a homomorphism of Poisson algebras between $(A, \{\,,\,\}_{T}, \circ_{T})$ and $(A, [\,,\,], \cdot)$,
			where $\{\,,\,\}_T$ and $\circ_T$ are defined by (for all $a, b \in A$)
			\begin{align*}
				\{a, b\}_T & := [T(a), b] + [a, T(b)] + \lambda [a, b] ,   \\
				a \circ_T b & := T(a) \cdot b + a \cdot T(b) + \lambda a \cdot b. 
			\end{align*}
			
		\item 
			$T + S$ (resp. $T - S$) is an $\mathcal{O}$-operator of weight 1 on $(A, [\,,\,], \cdot)$ associated to $(A$, $\{\,,\,\}_{+}$, $\circ_{+}$, $\ad_{[,]}$, $L_{\cdot})$ (resp. $(A$, $\{\,,\,\}_{-}$, $\circ_{-}$, $\ad_{[,]}$, $L_{\cdot})$).
	\end{enumerate}
\end{corollary}
\begin{proof}
	Note that $S: A \rightarrow A$ is equivalent of mass $\mu$ associated to $(A, [\,,\,], \cdot, \ad_{[,]}, L_\cdot)$ for all $\mu \in \mathbf{k}$.
	The conclusions follow from Proposition~\ref{prop:bae2amp} and Theorem~\ref{thm:opm}.
\end{proof}

\begin{corollary}
	Let $(A, [\,,\,], \cdot)$ be a Poisson algebra.
	Then $T: A \rightarrow A$ satisfies the following equalities for all $a, b \in A$
	\begin{align}
		[T(a), T(b)] - T([T(a), b] + [a, T(b)] + \lambda [a, b]) = (\pm\lambda - 1) [a, b] , \label{eq:bpal} \\
		T(a) \cdot T(b) - T(T(a) \cdot b + a \cdot T(b) + \lambda a \cdot b) = (\pm \lambda - 1) a \cdot b. \label{eq:bpaa}
	\end{align}
	if and only if $T \pm \id$ is a Rota-Baxter operator of weight $\lambda \mp 2$ on $(A, [\,,\,], \cdot)$.
\end{corollary}
\begin{proof}
	By Corollary~\ref{coro:eoqadj}, $T$ is an extended $\mathcal{O}$-operator of weight $\lambda$ with extension $S = \id$ of mass $(-1, \pm\lambda)$ on $(A, [\,,\,], \cdot)$ associated to $(A, [\,,\,], \cdot, \ad_{[,]}, L_\cdot)$ if and only if $T \pm \id$ is an $\mathcal{O}$-operator of weight 1 on $(A, [\,,\,], \cdot)$ associated to $(A, \{\,,\,\}_{\pm}, \circ_{\pm}, \ad_{[,]}, L_{\cdot})$, where
	\begin{equation*}
		\{a, b\}_{\pm} = (\lambda \mp 2) [a, b], \;\;
		a \circ_{\pm} b = (\lambda \mp 2) a \cdot b, \;\;
		\forall a, b \in A.
	\end{equation*}
	The latter is equivalent to the condition that $T \pm \id$ is an $\mathcal{O}$-operator of weight $\lambda \mp 2$ on $(A, [\,,\,], \cdot)$ associated to $(A, [\,,\,], \cdot, \ad_{[,]}, L_{\cdot})$, which completes the proof.
\end{proof}

\begin{remark}
	If $T$ satisfies Eqs.~\eqref{eq:bpal} and \eqref{eq:bpaa} for $\lambda = 0$, then $(A, [\,,\,], \cdot, T)$ is called a {\bf Baxter Poisson algebra}.
	A Baxter Poisson algebra naturally yields post-Poisson algebra structures.
	In fact, since $T \pm \id$ is a Rota-Baxter operator of weight $\mp 2$  on $(A, [\,,\,], \cdot)$, by Proposition~\ref{prop:o2pp}, $(A, \{\,,\,\}, \star, \circ, \succ)$ is a post-Poisson algebra, where $\{\,,\,\}, \star, \circ, \succ: A \otimes A \to A$ are defined for all $a, b \in A$ by 
	\begin{equation*}
		\{a, b\} = \mp 2 [a, b], \;
		a \star b = [(T \pm \id)(a), b], \;
		a \circ b = \mp 2 a \cdot b, \;
		a \succ b = (T \pm \id)(a) \cdot b.
	\end{equation*}
\end{remark}

\section{Tensor forms of extended \texorpdfstring{$\mathcal{O}$}{O}-operators and the extended Poisson Yang-Baxter equation}\label{sec:tensor}
In this section, we introduce the notion of the generalized Poisson Yang-Baxter equations and the extended Poisson Yang-Baxter equations, both of which generalize the Poisson Yang-Baxter equations. 
The latter is shown to be a special case of the former under the invariant condition. 
We then establish the relationship between extended $\mathcal{O}$-operators and the generalized Poisson Yang-Baxter equations, as well as the extended Poisson Yang-Baxter equations. 
In particular, the connection between the Poisson Yang-Baxter equations and extended $\mathcal{O}$-operators is also obtained.

\subsection{Extended \texorpdfstring{$\mathcal{O}$}{O}-operators and generalized Poisson Yang-Baxter equations}\label{sec:generalized}
In this subsection, we introduce the notion of generalized Poisson Yang-Baxter equations, which arise naturally in the study of Poisson bialgebras.
We then establish the relationship between extended $\mathcal{O}$-operators and the generalized Poisson Yang-Baxter equations.

\begin{definition}
	(\cite{ni2013poisson})
	A Poisson bialgebra $(A, [\ ,\ ], \cdot, \delta, \Delta)$ is called \textbf{coboundary} if there exists $r \in A \otimes A$ such that:
	\begin{align}
		\delta(a) & = (\id \otimes \ad_{[,]}(a) + \ad_{[,]}(a) \otimes \id )(r), \label{eq:pcbdl}         \\
		\Delta(a) & = (\id \otimes L_\cdot(a) - L_\cdot(a) \otimes \id)(r), \;\; \forall a \in A. \label{eq:pcbda}
	\end{align}
	A coboundary Poisson bialgebra is denoted by $(A, [\ ,\ ], \cdot, \delta_r, \Delta_r)$.
\end{definition}

\begin{theorem}\label{thm:cbd}
	{\rm (\cite[Theorem~2]{ni2013poisson})}
	Let $(A, [\ ,\ ], \cdot)$ be a Poisson algebra and $r \in A \otimes A$.
	Writing $r$ as $r = \Theta + \Lambda$ with $\Theta \in \mathrm{Sym}^2(A)$ and $\Lambda \in \mathrm{Alt}^2(A)$.
	Let $\delta: A \rightarrow A \otimes A$ and $\Delta: A \rightarrow A \otimes A$ be two linear maps defined by Eqs.~\eqref{eq:pcbdl} and \eqref{eq:pcbda} respectively.
	Then $(A^*, \delta^*, \Delta^*)$ is a Poisson algebra such that $(A, [\ ,\ ]_A, \cdot_A, \delta, \Delta)$ is a Poisson bialgebra if and only the following conditions are satisfied (for all $a \in A$):
	\begin{align}
		(\ad_{[,]}(a) \otimes \id + \id \otimes \ad_{[,]}(a))\Theta                                         & = 0, \label{eq:cpbi1}\\
		(L_{\cdot}(a) \otimes \id - \id \otimes L_{\cdot}(a))\Theta                                           & = 0, \label{eq:cpbi2}\\
		(\ad_{[,]}(a) \otimes \id \otimes \id + \id \otimes \ad_{[,]}(a) \otimes \id + \id \otimes \id \otimes \ad_{[,]}(a))\mathbf{C}(r)      & = 0, \label{eq:gepbe1}\\
		(L_{\cdot}(a) \otimes \id \otimes \id - \id \otimes \id \otimes L_{\cdot}(a))\mathbf{A}(r)                         & = 0, \label{eq:gepbe2}\\
		(\ad_{[,]}(a) \otimes \id \otimes \id)\mathbf{A}(r) - (\id \otimes L_{\cdot}(a) \otimes \id - \id \otimes \id \otimes L_{\cdot}(a))\mathbf{C}(r) & = 0, \label{eq:gepbe3}
	\end{align}
	where $\mathbf{C}(r) := [r_{12}, r_{13}] + [r_{13}, r_{23}] + [r_{12}, r_{23}]$ and $\mathbf{A}(r) := r_{12} \cdot r_{13} + r_{13} \cdot r_{23} - r_{23} \cdot r_{12}$.
\end{theorem}

\begin{remark}\label{rmk:gped}
	As shown in \cite{ni2013poisson}, $(A^*, \delta^*, \Delta^*)$ is a Poisson algebra if and only if Eqs.~\eqref{eq:gepbe1}-\eqref{eq:gepbe3} hold, provided that Eqs.~\eqref{eq:cpbi1}-\eqref{eq:cpbi2} are satisfied. 
\end{remark}

Skew-symmetric solutions of the Poisson Yang-Baxter equation were studied in \cite{ni2013poisson}, and solutions with invariant symmetric part were investigated in \cite{lin2026quasitriangular}; both are special cases of Eqs.~\eqref{eq:cpbi1}--\eqref{eq:cpbi2}. In the present paper, we retain the invariant condition and study a generalized version of the Poisson Yang-Baxter equation.

\begin{definition}
	(\cite{ni2013poisson})
	Let $(A, [\ ,\ ], \cdot)$ be a Poisson algebra and $r \in A \otimes A$.
	Then $r$ is called a solution of the \textbf{Poisson Yang-Baxter equation (PYBE)} in $(A, [\ ,\ ], \cdot)$ if 
	$\mathbf{C}(r) = \mathbf{A}(r) = 0$.
\end{definition}

\begin{definition}
	Let $(A, [\,,\,], \cdot)$ be a Poisson algebra.
	An element $r \in A \otimes A$ is called a solution of the \textbf{generalized Poisson Yang-Baxter equation (GPYBE)} in $(A, [\,,\,], \cdot)$ if for all $a \in A$, Eqs.~\eqref{eq:gepbe1}-\eqref{eq:gepbe3} hold.
\end{definition}

\begin{definition}[\cite{lin2026quasitriangular}]
	Let $(A, [\ ,\ ], \cdot)$ be a Poisson algebra and $r \in A \otimes A$.
	Then $r$ is called \textbf{$(\ad, L)$-invariant} if for all $a \in A$, the following equalities hold
	\begin{align}
		(\ad_{[,]}(a) \otimes \id + \id \otimes \ad_{[,]}(a))(r) &= 0, \label{eq:adiv} \\
		(L_\cdot(a) \otimes \id - \id \otimes L_\cdot(a))(r) &= 0. \label{eq:liv}
	\end{align}
\end{definition}

Clearly, solutions of the GPYBE with $(\ad, L)$-invariant symmetric part give rise to coboundary Poisson bialgebras. 
\begin{corollary}
	Let $(A, [\ ,\ ], \cdot)$ be a Poisson algebra and $r \in A \otimes A$.
	 Writing $r$ as $r = \Theta + \Lambda$ with $\Theta \in \mathrm{Sym}^2(A)$ and $\Lambda \in \mathrm{Alt}^2(A)$.
	Suppose that $\Theta$ is $(\ad, L)$-invariant,
	then $r$ is a solution of the GPYBE if and only if $\tau(r)$ is a solution of the GPYBE in $(A, [\ ,\ ], \cdot)$.
\end{corollary}
\begin{proof}
	Define linear maps $\delta_r$ and $\Delta_r$ by Eqs.~\eqref{eq:pcbdl} and \eqref{eq:pcbda} through $r$, respectively.
	Note that $\delta_{\tau(r)} = -\delta_{r}$ and $\Delta_{\tau(r)} = - \Delta_r$, and that $(A^*,- \delta_{r}^*, -\Delta_{r}^*)$ is a Poisson algebra if and only if $(A^*, \delta_{r}^*, \Delta_{r}^*)$ is a Poisson algebra.
	Remark~\ref{rmk:gped} then implies that $\tau(r)$ is a solution of the GPYBE if and only if $r$ is a solution of the GPYBE.
\end{proof}

Let $A$ be a vector space.
Any $r \in A \otimes A$ can be identified as maps from $A^*$ to $A$, which we denote by $r_{+}: A^* \rightarrow A$ and $r_{-}: A^* \rightarrow A$, respectively and explicitly,
\begin{equation*}
	\langle r_{+}(x^*), y^*\rangle = - \langle x^*, r_{-}(y^*)\rangle = \langle r, x^* \otimes y^*\rangle, \;\; \forall x^*, y^* \in A^*.
\end{equation*}
Clearly, $-r_{-} = r_{+}^* = (\tau(r))_{+}$.
Note that $r$ is called the \textbf{2-tensor form} of a linear map $\varphi: A^* \rightarrow A$ if $r_{+} = \varphi$.
Moreover, the bracket and multiplication on $A^*$ defined by Eqs.~\eqref{eq:pcbdl}-\eqref{eq:pcbda} (as the duals) are respectively given by
\begin{align}
	[x^*, y^*]_r  & = -\ad_{[,]}^*(r_{+}(x^*)) y^* +\ad_{[,]}^*(r_{-}(y^*)) x^* , \label{eq:pcbdr1}             \\
	x^* \cdot_r y^* & = L^*_{\cdot} (r_{+}(x^*)) y^* + L^*_{\cdot} (r_{-}(y^*)) x^*, \;\; \forall x^*,y^* \in A^*. \label{eq:pcbdr2}
\end{align}
Writing $r$ as $r = \Theta + \Lambda$ with $\Theta \in \mathrm{Sym}^2(A)$ and $\Lambda \in \mathrm{Alt}^2(A)$, i.e., 
\begin{equation*}
	\Theta = \frac{r + \tau(r)}{2}, \;\; \Lambda = \frac{r - \tau(r)}{2}
\end{equation*}
Immediately, $\tau(r) = \Theta - \Lambda$, and
\begin{equation*}
	\Theta_{+} = - \Theta_{-} = \frac{r_{+} - r_{-}}{2}, \;\; \Lambda_{+} = \Lambda_{-} = \frac{r_{+} + r_{-}}{2},
\end{equation*}
that is, $\Theta_{+}$ and $\Lambda_{+}$ are antisymmetrizer and symmetrizer of $r_{\pm}$ respectively.

\begin{lemma}\label{lem:sinvb}
	Let $(A, [\,,\,], \cdot)$ be a Poisson algebra and $r \in A \otimes A$.
	Writing $r$ as $r = \Theta + \Lambda$ with $\Theta \in \mathrm{Sym}^2(A)$ and $\Lambda \in \mathrm{Alt}^2(A)$.
	Then the following conditions are equivalent.
	\begin{enumerate}[label=(\roman*)]
		\item
		$\Theta$ is $(\ad, L)$-invariant.
		
		\item
		$\Theta_{+}: A^* \rightarrow A$ is $A$-invariant of mass 1 associated to $(A^*, -\ad_{[,]}^*, L_\cdot^*)$.
		
		\item
		$\Theta_{+}: A^* \rightarrow A$ is balanced associated to $(A^*, -\ad_{[,]}^*, L_\cdot^*)$.
	\end{enumerate}
\end{lemma}
\begin{proof}
	It follows from \cite[Lemma~2.14]{lin2026quasitriangular}.
\end{proof}

\begin{proposition}\label{prop:eocad2gp}
	Let $(A, [\,,\,], \cdot)$ be a Poisson algebra and $r \in A \otimes A$.
	Writing $r$ as $r = \Theta + \Lambda$ with $\Theta \in \mathrm{Sym}^2(A)$ and $\Lambda \in \mathrm{Alt}^2(A)$.
	Suppose that $\Theta$ is $(\ad, L)$-invariant.
	If $\Lambda_{+}$ is an extended $\mathcal{O}$-operator of weight $0$ with extension $\Theta_{+}$ of mass $(\kappa, 0)$ on $(A, [\,,\,], \cdot)$ associated to $(A^*, -\ad_{[,]}^*, L_\cdot^*)$, then $(A, \delta^*, \Delta^*)$ is a Poisson algebra on $A^*$ and $r$ is a solution of the GPYBE in $(A, [\,,\,], \cdot)$.
\end{proposition}
\begin{proof}
	By Lemma~\ref{lem:sinvb}, $\Theta_{+}$ is a balanced $A$-homomorphism associated to $(A^*, -\ad_{[,]}^*, L_\cdot^*)$.
	Hence, for all $x^*, y^* \in A^*$,
	\begin{align*}
		[x^*, y^*]_r & = -\ad_{[,]}^*(r_{+}(x^*)) y^* +\ad_{[,]}^*(r_{-}(y^*)) x^* = -\ad_{[,]}^*(\Lambda_{+}(x^*)) y^* +\ad_{[,]}^*(\Lambda_{+}(y^*)) x^*, \\
		x^* \cdot_r y^* & = L_{\cdot}^*(r_{+}(x^*)) y^* + L_{\cdot}^*(r_{-}(y^*)) x^* = L_{\cdot}^*(\Lambda_{+}(x^*)) y^* + L_{\cdot}^*(\Lambda_{+}(y^*)) x^*. 
	\end{align*}
	Theorem~\ref{thm:eo2ps} then implies that $(A^*, \delta^* = [\,,\,]_r, \Delta^* = \cdot_r)$ is a Poisson algebra.
	Moreover, by Remark~\ref{rmk:gped}, $r$ is a solution of the GPYBE in $(A, [\,,\,], \cdot)$.
\end{proof}

Let $T: V \rightarrow A$ be a linear map.
Through the identification $\Hom(V, A) \cong A \otimes V^* \subset (A \oplus V^*) \otimes (A \oplus V^*)$, we may regard $T$ as an element $r^{T} \in (A \oplus V^*) \otimes (A \oplus V^*)$ of the tensor product.
Concretely, for a basis $\{e_1, \cdots, e_n\}$ of $V$ with dual basis $\{e^1, \cdots, e^n\}$, we have $r^T = \sum_{i} T(e_i) \otimes e^i$.
The associated maps $r^{T}_{+}, r^{T}_{-}: A^* \oplus V \rightarrow A \oplus V^*$ are then explicitly given by
\begin{equation*}
	r^T_{+}(x^* + u) = T^*(x^*), \;\;
	r^T_{-}(x^* + u) = - T(u), \;\;
	\forall x^* \in A^*, \; u \in V.
\end{equation*}

\begin{theorem}\label{thm:gl2gp}
	Let $(A, [\,,\,], \cdot)$ be a Poisson algebra, $(V, \rho, \zeta)$ be a module of $(A, [\,,\,], \cdot)$, and $(\mathfrak{A} := A \ltimes_{-\rho^*, \zeta^*} V^*, \courant{\,,\,}, \bullet)$ be the semi-direct product of $(A, [\,,\,], \cdot)$ and $(V^*, -\rho^*, \zeta^*)$.
	Let $T: V \rightarrow A$ be a linear map.
	Then $\hat{r} := r^T - \tau(r^T)$ is a skew-symmetric solution of the GPYBE in $(A \ltimes_{-\rho^*, \zeta^*} V^*, \courant{\,,\,}, \bullet)$ if and only if (for all $a \in A$ and $u, v \in V$)
	\begin{align}
		[a, \mathcal{L}_{T}(u, v)] - \mathcal{L}_{T}(\rho(a)u, v) - \mathcal{L}_{T}(u, \rho(a)v)  & = 0, \label{eq:glgp1} \\
		\rho(\mathcal{L}_{T}(u, v))w + \rho(\mathcal{L}_{T}(v, w))u + \rho(\mathcal{L}_{T}(w, u))v & = 0, \label{eq:glgp2} \\
		a \cdot \mathcal{A}_T(u, v) - \mathcal{A}_T(\zeta(a)u, v)                 & = 0, \label{eq:glgp3} \\
		\zeta(\mathcal{A}_{T}(u, v))w - \zeta(\mathcal{A}_{T}(w, u))v & = 0, \label{eq:glgp4} \\
		[a, \mathcal{A}_T(u, v)] - \mathcal{L}_T(\zeta(a)u, v) + \mathcal{L}_T(u, \zeta(a)v)    & = 0, \label{eq:glgp5} \\
		\mathcal{A}_T(\rho(a)u, v) - a \cdot \mathcal{L}_T(u, v) + \mathcal{L}_T(u, \zeta(a)v)   & = 0, \label{eq:glgp6} \\
		\rho(\mathcal{A}_T(u, v)) w - \zeta(\mathcal{L}_T(v, w)) u + \zeta(\mathcal{L}_T(w, u)) v & = 0, \label{eq:glgp7} 
	\end{align}
	where $\mathcal{L}_{T}$ and $\mathcal{A}_{T}$ are respectively defined by
	\begin{align*}
		\mathcal{L}_{T}(u, v) & := [T(u), T(v)] - T(\rho(T(u))v - \rho(T(v))u),   \\
		\mathcal{A}_{T}(u, v) & := T(u) \cdot T(v) - T(\zeta(T(u))v + \zeta(T(v))u).
	\end{align*}
\end{theorem}
\begin{proof}
	Let $\{e_1, \cdots, e_n\}$ be a basis of $V$ and $\{e^1, \cdots, e^n\}$ be the dual basis.
	Then
	\begin{equation*}
		\hat{r} := r^{T} - \tau(r^T) = \sum_{i} \big( T(e_i) \otimes e^i - e^i \otimes T(e_i) \big).
	\end{equation*}
	Thus,
	\begin{align*}
		\mathbf{C}(\hat{r}) & =\sum_{i, j}\big( \mathcal{L}_{T}(e_i, e_j) \otimes e^i \otimes e^j + e^i \otimes e^j \otimes \mathcal{L}_T(e_i, e_j) - e^i \otimes \mathcal{L}_T(e_i, e_j) \otimes e^j\big),  \\
		\mathbf{A}(\hat{r}) & = \sum_{i, j}\big( \mathcal{A}_{T}(e_i, e_j) \otimes e^i \otimes e^j + e^i \otimes e^j \otimes \mathcal{A}_T(e_i, e_j) + e^i \otimes \mathcal{A}_T(e_i, e_j) \otimes e^j \big).
	\end{align*}
	Note that $(\ad_{\courant{,}}(a + \xi^*) \otimes \id \otimes \id + \id \otimes \ad_{\courant{,}}(a + \xi^*) \otimes \id + \id \otimes \id \otimes \ad_{\courant{,}}(a + \xi^*))\mathbf{C}(\hat{r}) = 0$ holds for all $a \in A$ and $\xi^* \in V^*$ if and only if
	\begin{align*}
		(\ad_{\courant{,}}(e_{k}) \otimes \id \otimes \id + \id \otimes \ad_{\courant{,}}(e_{k}) \otimes \id + \id \otimes \id \otimes \ad_{\courant{,}}(e_{k}))\mathbf{C}(\hat{r}) &= 0, \\
		(\ad_{\courant{,}}(e^k) \otimes \id \otimes \id + \id \otimes \ad_{\courant{,}}(e^k) \otimes \id + \id \otimes \id \otimes \ad_{\courant{,}}(e^k))\mathbf{C}(\hat{r}) &= 0,
	\end{align*}
	hold for all $e_{k} \in A$ and $e^k \in V^*$ ($k =1,\cdots n$). 
	For all $e_{k} \in A$ and $s, t, l, k \in \{1, \cdots n\}$, we have
	\begin{align*}
		&(\ad_{\courant{,}}(e_{k}) \otimes \id \otimes \id + \id \otimes \ad_{\courant{,}}(e_{k}) \otimes \id + \id \otimes \id \otimes \ad_{\courant{,}}(e_{k}))\mathbf{C}(\hat{r}) \\
		&= \sum_{i, j}\big( [e_{k}, \mathcal{L}_{T}(e_i, e_j)] \otimes e^i \otimes e^j - e^i \otimes e^j \otimes \mathcal{L}_T(\rho(e_{k}) e_i, e_j) + e^i \otimes \mathcal{L}_T(e_i, \rho(e_{k}) e_j) \otimes e^j \big) \\
		&\quad + \sum_{i, j}\big( -\mathcal{L}_{T}(\rho(e_{k})e_i, e_j) \otimes e^i \otimes e^j - e^i \otimes e^j \otimes \mathcal{L}_T(e_i, \rho(e_{k}) e_j) - e^i \otimes [e_{k}, \mathcal{L}_T(e_i, e_j)] \otimes e^j \big) \\
		&\quad + \sum_{i, j}\big( -\mathcal{L}_{T}(e_i, \rho(e_{k})e_j) \otimes e^i \otimes e^j + e^i \otimes e^j \otimes [e_{k}, \mathcal{L}_T(e_i, e_j)] + e^i \otimes \mathcal{L}_T(e_i, \rho(e_{k}) e_j) \otimes e^j \big) \\
		&= \sum_{i,j}\big( [e_{k}, \mathcal{L}_{T}(e_i, e_j)] - \mathcal{L}_{T}(\rho(e_{k})e_i, e_j) - \mathcal{L}_{T}(e_i, \rho(e_{k})e_j ) \big) \otimes e^i \otimes e^j \\
		&\quad + \sum_{i,j} e^i \otimes e^j \otimes \big( [e_{k}, \mathcal{L}_{T}(e_i, e_j)] - \mathcal{L}_{T}(\rho(e_{k})e_i, e_j) - \mathcal{L}_{T}(e_i, \rho(e_{k})e_j ) \big) \\
		&\quad - \sum_{i, j} e^i \otimes \big( [e_{k}, \mathcal{L}_{T}(e_i, e_j)] - \mathcal{L}_{T}(\rho(e_{k})e_i, e_j) - \mathcal{L}_{T}(e_i, \rho(e_{k})e_j ) \big) \otimes e^j,
	\end{align*}
	and 
	\begin{align*}
		&\langle (\ad_{\courant{,}}(e^k) \otimes \id \otimes \id + \id \otimes \ad_{\courant{,}}(e^k) \otimes \id + \id \otimes \id \otimes \ad_{\courant{,}}(e^k))\mathbf{C}(\hat{r}), e_s \otimes e_t \otimes e_l \rangle \\
		&= \langle \sum_{i, j} (\rho(\mathcal{L}_{T}(e_i, e_j))e^k \otimes e^i \otimes e^j - e^i \otimes \rho(\mathcal{L}_{T}(e_i, e_j))e^k \otimes e^j + e^i \otimes e^j \otimes \rho(\mathcal{L}_{T}(e_i, e_j))e^k, e_s \otimes e_t \otimes e_k \rangle \\
		&= \langle e^k, \rho(\mathcal{L}_{T}(e_t, e_k))e_s - \rho(\mathcal{L}_{T}(e_s, e_k))e_t + \rho(\mathcal{L}_{T}(e_s, e_t))e_k \rangle \\
		&= \langle e^k, \rho(\mathcal{L}_{T}(e_t, e_k))e_s + \rho(\mathcal{L}_{T}(e_k, e_s))e_t + \rho(\mathcal{L}_{T}(e_s, e_t))e_k \rangle.
	\end{align*}
	Hence, $(\ad_{\courant{,}}(a + \xi^*) \otimes \id \otimes \id + \id \otimes \ad_{\courant{,}}(a + \xi^*) \otimes \id + \id \otimes \id \otimes \ad_{\courant{,}}(a + \xi^*))\mathbf{C}(\hat{r}) = 0$ holds for all $a \in A$ and $\xi^* \in V^*$ if and only if
	\begin{align*}
		[e_{k}, \mathcal{L}_{T}(e_i, e_j)] - \mathcal{L}_{T}(\rho(e_{k})e_i, e_j) - \mathcal{L}_{T}(e_i, \rho(e_{k})e_j ) &= 0, \\
		\rho(\mathcal{L}_{T}(e_i, e_j))e_k + \rho(\mathcal{L}_{T}(e_j, e_k))e_i + \rho(\mathcal{L}_{T}(e_k, e_i))e_j &= 0,
	\end{align*}
	hold for all $a \in A$ and $i, j, k \in \{1, \cdots n\}$, 
	which is equivalent to the condition that Eqs.~\eqref{eq:glgp1} and \eqref{eq:glgp2} hold for all $a \in A$ and $u, v, w \in V$.

	Similarly, $(L_{\bullet}(a + \xi^*) \otimes \id \otimes \id - \id \otimes \id \otimes L_{\bullet}(a + \xi^*))\mathbf{A}(\hat{r}) = 0$ holds for all $a \in A$ and $\xi^* \in V^*$ if and only if Eqs.~\eqref{eq:glgp3}-\eqref{eq:glgp4} hold for all $a \in A$ and $u, v, w \in V$.
	Finally, $(\ad_{\courant{,}}(a + \xi^*) \otimes \id \otimes \id)\mathbf{A}(\hat{r}) = (\id \otimes L_{\bullet}(a + \xi^*) \otimes \id - \id \otimes \id \otimes L_{\bullet}(a + \xi^*))\mathbf{C}(\hat{r})$ holds for all $a \in A$ and $\xi^* \in V^*$ if and only if Eqs.~\eqref{eq:glgp5}-\eqref{eq:glgp7} hold for all $a \in A$ and $u, v, w \in V$.
	The proof is complete.
\end{proof}

\begin{proposition}\label{prop:ex2gp}
	Let $(A, [\,,\,], \cdot)$ be a Poisson algebra, $(V, [\,,\,]_V, \cdot_V, \rho, \zeta)$ be an $A$-module Poisson algebra, and $S, T: V \to A$ two linear maps.
	Let $(A \ltimes_{-\rho^*, \zeta^*} V^*, \courant{\,,\,}, \bullet)$ be the semi-direct product of $(A, [\,,\,], \cdot)$ and $(V^*, -\rho^*, \zeta^*)$.
	If $S$ satisfies Eqs.~\eqref{eq:bal}, \eqref{eq:ainv} and \eqref{eq:eqvm} and $T: V \to A$ is an extended $\mathcal{O}$-operator of weight $\lambda$ with extension $S$ of mass $(\kappa, \mu)$ associated to $(V, [\,,\,]_V, \cdot_V, \rho, \zeta)$,
	then $r^T - \tau(r^T)$ is a skew-symmetric solution of the GPYBE in $(A \ltimes_{-\rho^*, \zeta^*} V^*, \courant{\,,\,}, \bullet)$ if and only if
	\begin{align}
		\lambda [a, T([u, v]_V)] - \lambda T(\rho(a)[u, v]_V)                                            & = 0, \label{eq:eogp1} \\
		\lambda \rho(T([u, v]_V))w + \lambda \rho(T([v, w]_V))u + \lambda \rho(T([w, u]_V))v                            & = 0, \label{eq:eogp2 }\\
		\lambda a \cdot T(u \cdot_V v) - \lambda T\big( \zeta(a)(u \cdot_V v) \big)                                 & = 0, \label{eq:eogp3} \\
		\lambda \zeta(T(u \cdot_V v))w - \lambda \zeta(T(w \cdot_V u))v                                       & = 0, \label{eq:eogp4} \\
		\lambda [a, T(u \cdot_V v) ] - \lambda T \big(\rho(a)(u \cdot_V v) \big) + \mu [a, S(u \cdot_V v) ] - \mu S \big(\rho(a)(u \cdot_V v) \big) & = 0, \label{eq:eogp5} \\
		\lambda a \cdot T([u, v]_V) -\lambda T\big(\zeta(a) [u, v]_V \big) + \mu a \cdot S([u, v]_V) - \mu S\big(\zeta(a) [u, v]_V \big) & = 0, \label{eq:eogp6} \\
		\lambda \rho(T(u \cdot_V v))w + \lambda \zeta(T([w, v]_V))u + + \lambda \zeta(T([w, u]_V))v                         & = 0, \label{eq:eogp7} 
	\end{align}
	for all $u, v, w \in V$.
\end{proposition}
\begin{proof}
	Since $T: V \to A$ is an extended $\mathcal{O}$-operator of weight $\lambda$ with extension $S$ of mass $(\kappa, \mu)$ associated to $(V, [\,,\,]_V, \cdot_V, \rho, \zeta)$, we have (for all $u, v \in V$)
	\begin{align*}
		\mathcal{L}_{T}(u, v) & := [T(u), T(v)] - T(\rho(T(u))v - \rho(T(v))u) = \lambda T([u, v]_V) + \kappa [S(u), S(v)] + \mu S([u, v]_V),        \\
		\mathcal{A}_{T}(u, v) & := T(u) \cdot T(v) - T(\zeta(T(u))v + \zeta(T(v))u) = \lambda T(u \cdot_V v) + \kappa S(u) \cdot S(v) + \mu S(u \cdot_V v).
	\end{align*}
	Since $S$ satisfies Eqs.~\eqref{eq:bal}-\eqref{eq:eqvm}, it follows that Eqs.~\eqref{eq:glgp1}-\eqref{eq:glgp7} and Eqs.~\eqref{eq:eogp1}-\eqref{eq:eogp7} are respectively equivalent.
	The conclusion then follows from Theorem~\ref{thm:gl2gp}.
\end{proof}

\begin{corollary}
	Let $(A, [\,,\,], \cdot)$ be a Poisson algebra.
	\begin{enumerate}[label=(\roman*)]
		\item
		   Let $(V, \rho, \zeta)$ be a module of $(A, [\,,\,], \cdot)$ and $(A \ltimes_{-\rho^*, \zeta^*} V^*, \courant{\,,\,}, \bullet)$ be the semi-direct product of $(A, [\,,\,], \cdot)$ and $(V^*, -\rho^*, \zeta^*)$.
		   If $S$ satisfies Eqs.~\eqref{eq:bal} and \eqref{eq:ainv}, and $T: V \to A$ is an extended $\mathcal{O}$-operator of weight $\lambda$ with extension $S$ of mass $(\kappa, 0)$ associated to $(V, \rho, \zeta)$,
		   then $r^T - \tau(r^T)$ is a skew-symmetric solution of the GPYBE in $(A \ltimes_{-\rho^*, \zeta^*} V^*, \courant{\,,\,}, \bullet)$.

		\item
			 Let $(A \ltimes_{-\ad_{[,]}^*, L_\cdot^*} A^*, \courant{\,,\,}, \bullet)$ be the semi-direct product of $(A, [\,,\,], \cdot)$ and $(A^*, -\ad_{[,]}^*, L_\cdot^*)$.
		   If $T: A \rightarrow A$ be an extended $\mathcal{O}$-operator of weight 0 with extension $\id: A \rightarrow A$ of mass $(\kappa, 0)$ associated to $(A, \ad_{[,]}, L_\cdot)$,
		   then $r^T - \tau(r^T)$ is a skew-symmetric solution of the GPYBE in $(A \ltimes_{-\ad_{[,]}^*, L_\cdot^*} A^*, \courant{\,,\,}, \bullet)$.
	\end{enumerate}
\end{corollary}
\begin{proof}
	Interpreting a module of $(A, [\,,\,], \cdot)$ as an $A$-module Poisson algebra endowed with the trivial Poisson algebra structure, the conclusion is a direct consequence of Proposition~\ref{prop:ex2gp}.
\end{proof}

\begin{example}\label{HesPois2}
Let $(A, [\,,\,], \cdot)$ be the 3-dimensional Poisson algebra given in Example~\ref{HesPois} with basis $\{e_1,e_2,e_3\}$ whose non-zero bracket and multiplication are given by Eq.~\eqref{HPCF1}. Define linear maps $ T: A \to A$ by
\begin{align*} 
T\left(e_{1}\right)&=2e_{1}, && T\left(e_{2}\right)=\frac{1}{2} e_{2}, 
&& T\left(e_{3}\right)=0. 
\end{align*}
 Then $T: A \rightarrow A$ is an extended $\mathcal{O}$-operator of weight 0 with extension $S = \id$ of mass $(1, 0)$ associated to $(A, \ad_{[,]}, L_\cdot)$.
	Let $\{e_1^*, e^*_2, e_3^*\}$ be the dual basis of $\{e_1,e_2,e_3\}$.
	Then 
	\begin{equation*}
		r^T - \tau(r^T) = 2e_1 \otimes e_1^* + \frac{1}{2}e_2 \otimes e_2^*  -2e_1^* \otimes e_1 - \frac{1}{2}e_2^* \otimes e_2
	\end{equation*}
	is a skew-symmetric solution of the GPYBE in $(A \ltimes_{-\ad_{[,]}^*, L_\cdot^*} A^*, \courant{\,,\,}, \bullet)$, where non-zero bracket $\courant{,}$ and multiplication $\bullet$ are given by
	\begin{equation}
	\courant{e_1,e_2} = e_3, \;\;  \courant{e_3^*,e_1}= e_2^*, \;\; \courant{e_3^*,e_2} = -e_1^*, \;\;
	e_1 \bullet e_2 = e_3, \;\; e_3^* \bullet e_1 = e_2^*, \;\; e_3^* \bullet e_2 = e_1^*. \label{HesPoisSP}
	\end{equation}

\end{example}

\subsection{The extended Poisson Yang-Baxter equations}\label{ss:epybe}

\begin{definition}\label{def:epybe}
	Let $(A, [\,,\,], \cdot)$ be a Poisson algebra and $\epsilon \in \mathbf{k}$.
	We say that $r\in A \otimes A$ is a solution of the \textbf{extended Poisson Yang-Baxter equation (EPYBE)} of mass $\epsilon$ in $(A, [\,,\,], \cdot)$ if
	\begin{align}
		\mathbf{C}(r) & = \epsilon [ r_{13} +r_{31}, r_{23} + r_{32} ], \label{eq:epybel}       \\
		\mathbf{A}(r) & = \epsilon ( r_{13} +r_{31}) \cdot ( r_{23} + r_{32} ). \label{eq:epybea}
	\end{align}
\end{definition}
 
\begin{remark}
	When $\epsilon = 0$ or $r$ is skew-symmetric, the EPYBE of mass $\epsilon$ coincides with the PYBE~\cite{ni2013poisson}.
	Later, we will show that solutions of the EPYBE whose symmetric part is $(\ad, L)$-invariant are also solutions of the GPYBE.
\end{remark}

\begin{remark}
	When the symmetric part of $r$ is $(\ad, L)$-invariant, a quick computation shows the EPYBE of mass $\epsilon$ in $(A, [\,,\,], \cdot)$ is equivalent to either of the following forms:
	\begin{equation*}
		\mathbf{C}(r) = \epsilon [r_{23} + r_{32}, r_{12} + r_{21}], \;\; \mathbf{A}(r) = \epsilon ( r_{23} +r_{32}) \cdot ( r_{12} + r_{21} )
	\end{equation*}
	or
	\begin{equation*}
		\mathbf{C}(r) = \epsilon [r_{12} + r_{21}, r_{13} + r_{31}], \;\; \mathbf{A}(r) = \epsilon ( r_{12} +r_{21}) \cdot ( r_{13} + r_{31} ).
	\end{equation*}
	Indeed, we show that 
	\begin{align*}
		&[r_{23} + r_{32}, r_{12} + r_{21}] \\
		&= \sum_{i,j} \big( a_j \otimes [a_i, b_j] \otimes b_i + a_j \otimes [b_i, b_j] \otimes a_i + b_j \otimes [a_i, a_j] \otimes b_i + b_j \otimes [b_i, a_j] \otimes a_i \big) \\
		&= - \sum_{i,j} \big( a_j \otimes a_i \otimes [b_i, b_j] + a_j \otimes b_i \otimes [a_i, b_j] + b_j \otimes a_i \otimes [b_i, a_j] + b_j \otimes b_i \otimes [a_i, a_j] \big) \\
		&= [ r_{13} +r_{31}, r_{23} + r_{32} ],
	\end{align*}
	and the other identities follow similarly.
\end{remark}

\begin{theorem}\label{thm:epybe-o}
	Let $(A, [\,,\,], \cdot)$ be a Poisson algebra and $r \in A \otimes A$.
	Writing $r$ as $r = \Theta + \Lambda$ with $\Theta \in \mathrm{Sym}^2(A)$ and $\Lambda \in \mathrm{Alt}^2(A)$.
	Suppose that $\Theta$ is $(\ad, L)$-invariant.
	Then $r$ is a solution of the EPYBE of mass $\frac{\kappa+1}{4}$ in $(A, [\,,\,], \cdot)$ if and only if $\Lambda_{+}: A^* \rightarrow A$ is an extended $\mathcal{O}$-operator of weight $0$ with extension $\Theta_{+}$ of mass $(\kappa, 0)$ on $(A, [\,,\,], \cdot)$ associated to $(A^*, -\ad_{[,]}^*, L_\cdot^*)$.
\end{theorem}
\begin{proof}
	For all $x^*, y^*, z^* \in A^*$, note that $\Theta_{+}: A^* \rightarrow A$ is a balance $A$-module homomorphism associated to $(A^*, -\ad_{[,]}^*, L_\cdot^*)$, we have
	\begin{align*}
		 & [ r_{+}(a^{*}) , r_{+}(b^{*})] - r_{+}(-\ad_{[,]}^{*}(r_{+}(a^{*})) b^{*}+\ad_{[,]}^{*}(r_{-}(b^{*})) a^{*})                                    \\
		 & =[ (\Lambda_{+}+\Theta_{+})(a^{*}) , (\Lambda_{+}+\Theta_{+})(b^{*})] - (\Lambda_{+}+\Theta_{+})(-\ad_{[,]}^{*}((\Lambda_{+}+\Theta_{+})(a^{*})) b^{*}+\ad_{[,]}^{*}((\Lambda_{+}-\Theta_{+})(b^{*})) a^{*}) \\
		 & =[\Lambda_{+}(x^*), \Lambda_{+}(y^*)] - \Lambda_{+}\big(-\ad_{[,]}^*(\Lambda_{+}(x^*))y^* + \ad_{[,]}^*(\Lambda_{+}(y^*))x^*\big) + [\Theta_{+}(x^*), \Theta_{+}(y^*)]      \\
		 & \quad + [\Theta_{+}(x^*), \Lambda_{+}(y^*)] + [\Lambda_{+}(x^*), \Theta_{+}(y^*)] - \Lambda_{+}\big(-\ad_{[,]}^*(\Theta_{+}(x^*))y^* - \ad_{[,]}^*(\Theta_{+}(y^*))x^*\big)         \\
		 & \quad - \Theta_{+}(-\ad_{[,]}^{*}((\Lambda_{+})(a^{*})) b^{*}+\ad_{[,]}^{*}((\Lambda_{+})(b^{*})) a^{*}) - \Theta_{+}(-\ad_{[,]}^{*}((\Theta_{+})(a^{*})) b^{*}+\ad_{[,]}^{*}((-\Theta_{+})(b^{*})) a^{*})     \\
		 & =[\Lambda_{+}(x^*), \Lambda_{+}(y^*)] - \Lambda_{+}\big(-\ad_{[,]}^*(\Lambda_{+}(x^*))y^* + \ad_{[,]}^*(\Lambda_{+}(y^*))x^*\big) + [\Theta_{+}(x^*), \Theta_{+}(y^*)].
	\end{align*}
	Therefore,
	\begin{align*}
		 & \langle [\Lambda_{+}(x^*), \Lambda_{+}(y^*)] - \Lambda_{+}\big(-\ad_{[,]}^*(\Lambda_{+}(x^*))y^* + \ad_{[,]}^*(\Lambda_{+}(y^*))x^*\big) - \kappa [\Theta_{+}(x^*), \Theta_{+}(y^*)], z^* \rangle \\
		 & =\langle [ r_{+}(a^{*}) , r_{+}(b^{*})] - r_{+}(-\ad_{[,]}^{*}(r_{+}(a^{*})) b^{*}+\ad_{[,]}^{*}(r_{-}(b^{*})) a^{*}) - (\kappa+1) [\Theta_{+}(x^*), \Theta_{+}(y^*)], z^*\rangle               \\
		 & =\langle [r_{13}, r_{23}] + [r_{12}, r_{23}] + [r_{12}, r_{13}] - \frac{(\kappa+1)}{4}[r_{13}+r_{31}, r_{23}+r_{32}], x^* \otimes y^* \otimes z^*\rangle.
	\end{align*}
	Similarly, we show that
	\begin{align*}
		 & \langle \Lambda_{+}(x^*) \cdot \Lambda_{+}(y^*) - \Lambda_{+}\big(L_\cdot^*(\Lambda_{+}(x^*))y^* + L_\cdot^*(\Lambda_{+}(y^*))x^* \big) - \kappa \Theta_{+}(x^*) \cdot \Theta_{+}(y^*) , z^* \rangle \\
		 & = \langle r_{12} \cdot r_{13} - r_{12} \cdot r_{23} + r_{13} \cdot r_{23} - \frac{(\kappa+1)}{4} ( r_{13} +r_{31}) \cdot ( r_{23} + r_{32} ), x^* \otimes y^* \otimes z^* \rangle.
	\end{align*}
	Hence, $r$ is a solution of the EPYBE of mass $\frac{\kappa+1}{4}$ in $(A, [\,,\,], \cdot)$ if and only if $\Lambda_{+}: A^* \rightarrow A$ is an extended $\mathcal{O}$-operator of weight $0$ with extension $\Theta_{+}$ of mass $(\kappa, 0)$ on $(A, [\,,\,], \cdot)$ associated to $(A^*, -\ad_{[,]}^*, L_\cdot^*)$.
\end{proof}

\begin{corollary}
	Let $(A, [\,,\,], \cdot)$ be a Poisson algebra.
	If $r \in A \otimes A$ is a solution of the EPYBE of mass $\epsilon$ in $(A, [\,,\,], \cdot)$ with $(\ad, L)$-invariant symmetric part, then $r$ is a solution of the GPYBE in $(A, [\,,\,], \cdot)$.
\end{corollary}
\begin{proof}
	Write $r$ as $r = \Theta + \Lambda$, where $\Theta \in \mathrm{Sym}^2(A)$ and $\Lambda \in \mathrm{Alt}^2(A)$.
	Since $\Theta$ is $(\ad, L)$-invariant, Theorem~\ref{thm:epybe-o} implies that $\Lambda_{+}: A^* \rightarrow A$ is an extended $\mathcal{O}$-operator of weight $0$ with extension $\Theta_{+}$ of mass $(4\epsilon - 1, 0)$ on $(A, [\,,\,], \cdot)$ associated to $(A^*, -\ad_{[,]}^*, L_\cdot^*)$.
	Proposition~\ref{prop:eocad2gp} then shows that $r$ is a solution of the GPYBE in $(A, [\,,\,], \cdot)$.
\end{proof}

\begin{proposition}\label{prop:peqc}
	Let $(A, [\,,\,], \cdot)$ be a Poisson algebra and $r \in A \otimes A$.
	Writing $r$ as $r = \Theta + \Lambda$ with $\Theta \in \mathrm{Sym}^2(A)$ and $\Lambda \in \mathrm{Alt}^2(A)$.
	Suppose that $\Theta$ is $(\ad, L)$-invariant.
	Then the following conditions are equivalent.
	\begin{enumerate}[label=(\roman*)]
		\item\label{it:reqd1}
		   $r$ is a solution of the PYBE in $(A, [\,,\,], \cdot)$.

		\item\label{it:reqd2}
		   $\Lambda_{+}$ is an extended $\mathcal{O}$-operator of weight $0$ with extension $\Theta_{+}$ of mass $(-1, 0)$ associated to $(A^*, -\ad_{[,]}^*, L_\cdot^*)$.

		\item\label{it:reqd3}
		   $r_{+}$ (resp. $r_{-}$) is an $\mathcal{O}$-operator of weight $1$ on $(A, [\,,\,], \cdot)$ associated to $A$-module Poisson algebra $(A^*, \{\,,\,\}_{+}, \circ_{+}, -\ad_{[,]}^*, L_\cdot^*)$ (resp. $(A^*, \{\,,\,\}_{-}, \circ_{-}, -\ad_{[,]}^*, L_\cdot^*)$), where $\{\,,\,\}_{+}$ and $\circ_{+}$ (resp. $\{\,,\,\}_{-}$ and $\circ_{-}$) are defined by 
		   \begin{align}
		   	\{x^*, y^*\}_{+} & = 2 \ad_{[,]}^*(\Theta_{+}(x^*)) y^* \quad \text{(resp. $\{x^*, y^*\}_{-} = - 2 \ad_{[,]}^*(\Theta_{+}(x^*)) y^*$)}, \label{eq:dbd} \\
		   	x^* \circ_{+} y^* & = - 2 L_\cdot^*(\Theta_{+}(x^*)) y^* \quad\;\;\; \text{(resp. $x^* \circ_{-} y^* = 2 L_\cdot^*(\Theta_{+}(x^*)) y^*$)}. \label{eq:dmd}
		   \end{align}

		\item\label{it:reqd4}
		   $(A^*, [\,,\,]_r, \cdot_r)$ is a Poisson algebra, where $[\,,\,]_r$ and $\cdot_r$ are defined by Eqs.~\eqref{eq:pcbdr1} and \eqref{eq:pcbdr2} respectively, and $r_{+}$ (resp. $r_{-}$) is homomorphism from $(A^*, [\,,\,]_r, \cdot_r)$ to $(A, [\,,\,], \cdot)$.
	\end{enumerate}
\end{proposition}
\begin{proof}
	\ref{it:reqd1} $\Longleftrightarrow$ \ref{it:reqd2}.
	Note that $r$ is a solution of the PYBE in $(A, [\,,\,], \cdot)$ if and only if $r$ is a solution of the EPYBE of mass $0$ in $(A, [\,,\,], \cdot)$.
	Then by Theorem~\ref{thm:epybe-o}, we show that $r$ is a solution of the PYBE in $(A, [\,,\,], \cdot)$ if and only if $\Lambda_{+}$ is an extended $\mathcal{O}$-operator of weight $0$ with extension $\Theta_{+}$ of mass $(-1, 0)$ associated to $(A^*, -\ad_{[,]}^*, L_\cdot^*)$.

	\ref{it:reqd2} $\Longleftrightarrow$ \ref{it:reqd3} $\Longleftrightarrow$ \ref{it:reqd4}.
	By Lemma~\ref{lem:sinvb}, $\Theta_{+}: A^* \rightarrow A$ is a balanced $A$-module homomorphism associated to $(A^*, -\ad_{[,]}^*, L_\cdot^*)$.
	Moreover, for all $x^*, y^* \in A^*$, we have
	\begin{align*}
		-\ad_{[,]}^*(\Lambda_{+}(x^*))y^* + \ad_{[,]}^*(\Lambda_{+}(y^*))x^* &= -\ad_{[,]}^*(r_{+}(x^*)) y^* +\ad_{[,]}^*(r_{-}(y^*)) x^*, \\
		L_{\cdot}^*(\Lambda_{+}(x^*))y^* + L_{\cdot}^*(\Lambda_{+}(y^*))x^* &= L_{\cdot}^*(r_{+}(x^*)) y^* + L_{\cdot}^*(r_{-}(y^*)) x^*.
	\end{align*}
	The desired conclusion now follows immediately from Corollary~\ref{coro:bah2e}.
\end{proof}

\begin{remark}
	\begin{enumerate}[nosep, label=(\alph*)]
		\item The equivalence \ref{it:reqd1} $\Longleftrightarrow$ \ref{it:reqd4} was previously demonstrated in \cite[Theorem~2.15]{lin2026quasitriangular} in the study of quasi-triangular Poisson bialgebras.
		Proposition \ref{prop:peqc} offers a new proof of this equivalence by means of extended $\mathcal{O}$-operators.
		
		\item 
		In particular, if $r$ is skew-symmetric, i.e., $\Theta = 0$, then $r$ is a solution of the PYBE in $(A, [\,,\,], \cdot)$ if and only if $r_{+}$ is an $\mathcal{O}$-operator of weight 0 on $(A, [\,,\,], \cdot)$ associated to $(A^*, -\ad_{[,]}^*, L_\cdot^*)$, as shown in~\cite{lin2023differential, LBS2020}. 
		This recovers the known $\mathcal{O}$-operator characterization.
	\end{enumerate}

\end{remark}

\begin{corollary}
	Let $(A, [\,,\,], \cdot)$ be a Poisson algebra and $r \in A \otimes A$ be a solution of the PYBE in $(A, [\,,\,], \cdot)$ whose symmetric part $\Theta$ is $(\ad, L)$-invariant.
	Define the following binary operations for all $x^*, y^* \in A^*$:
	\begin{align*}
		\{x^*, y^*\} &= 2 \ad_{[,]}^*(\Theta_{+}(x^*)) y^*, \quad\quad
		x^* \star y^* = -\ad_{[,]}^*(r_{+}(x^*))y^*, \\
		x^* \circ y^* &= - 2 L_\cdot^*(\Theta_{+}(x^*)) y^*, \quad\quad\;\;
		x^* \succ y^* = L_\cdot^*(r_{+}(x^*))y^*.
	\end{align*}
	Then $(A^*, \{\,,\,\}, \star, \circ, \succ)$ is a post-Poisson algebra.
	In particular, if $r_{+}$ is invertible, then there exists a post-Poisson algebra structure defined on $A$, whose associated Poisson algebra is precisely $(A, [\,,\,], \cdot)$.
\end{corollary}
\begin{proof}
	By Proposition~\ref{prop:peqc}, $r_{+}$ is an $\mathcal{O}$-operator of weight $1$ on $(A, [\,,\,], \cdot)$ associated to the $A$-module Poisson algebra $(A^*, \{\,,\,\}_{+}, \circ_{+}, -\ad_{[,]}^*, L_\cdot^*)$, where $\{\,,\,\}_{+}$ and $\circ_{+}$ are defined by Eqs.~\eqref{eq:dbd} and \eqref{eq:dmd}.
	The conclusions now follow from Proposition~\ref{prop:o2pp} and Theorem~\ref{thm:ippeq}.
\end{proof}

\subsection{Extended \texorpdfstring{$\mathcal{O}$}{O}-operators and EPYBE on quadratic Poisson algebras}\label{sec:quadratic}

\begin{definition}
	A \textbf{quadratic Poisson algebra} is a quadruple $(A, [\,,\,], \cdot, \mathcal{B})$, where $(A, [\,,\,], \cdot)$ is a Poisson algebra and $\mathcal{B}$ is a nondegenerate symmetric bilinear form on $A$ such that
	\begin{equation*}
		\mathcal{B}([a, b], c) = \mathcal{B}(a, [b, c]), \;\;
		\mathcal{B}(a \cdot b, c) = \mathcal{B}(a, b \cdot c), \;\;
		\forall a, b, c \in A.
	\end{equation*}
\end{definition}

\begin{definition}
	Let $(A, [\,,\,], \cdot)$ be a Poisson algebra and $\mathcal{B}$ be a nondegenerate symmetric bilinear form on $A$.
	A linear map $T: A \rightarrow A$ is called {\bf self-adjoint} (resp. {\bf skew-adjoint}) with respect to $\mathcal{B}$ if for all $a, b \in A$
	\begin{equation*}
		\mathcal{B}(T(a), b) = \mathcal{B}(a, T(b)) \quad \text{(resp. $\mathcal{B}(T(a), b) = -\mathcal{B}(a, T(b))$)}.
	\end{equation*}
\end{definition}

\begin{lemma}\label{lem:sa2s}
	Let $(A, [\,,\,], \cdot, \mathcal{B})$ be a quadratic Poisson algebra and $I_\mathcal{B}$ be the induced linear isomorphism by $\mathcal{B}$.
	Then $T: A \to A$ is self-adjoint (resp. skew-adjoint) with respect to $\mathcal{B}$ if and only if the 2-tensor form of $T I_\mathcal{B}: A^* \to A$ is symmetric (resp. skew-symmetric).
\end{lemma}
\begin{proof}
	It is enough to prove the self-adjoint case, since the skew-adjoint case is analogous.
	Let $r$ be the 2-tensor form of $T I_{\mathcal{B}}$, i.e., $r_{+} = T I_{\mathcal{B}}$.
	Then, for all $x^*, y^* \in A^*$, we have
	\begin{align*}
		& \langle r - \tau(r), x^* \otimes y^* \rangle = \langle y^*, r_{+}(x^*) \rangle - \langle x^*, r_{+}(y^*) \rangle = \mathcal{B}(I_{\mathcal{B}}(y^*), r_{+}(x^*)) - \mathcal{B}(I_{\mathcal{B}}(x^*), r_{+}(y^*)) \\
		=\; & \mathcal{B}(I_{\mathcal{B}}(y^*), T I_{\mathcal{B}}(x^*)) - \mathcal{B}(I_{\mathcal{B}}(x^*), T I_{\mathcal{B}}(y^*)) = \mathcal{B}(I_{\mathcal{B}}(y^*), T I_{\mathcal{B}}(x^*)) - \mathcal{B}(T I_{\mathcal{B}}(y^*), I_{\mathcal{B}}(x^*)).
	\end{align*}
	Hence, $r$ is symmetric if and only if $T$ is self-adjoint.
\end{proof}

Let $A$ be a vector space and $\mathcal{B}$ be a nondegenerate bilinear form on $A$.
Denote by $I_\mathcal{B}: A^* \rightarrow A$ the induced linear isomorphism defined by
\begin{equation*}
	\langle I_\mathcal{B}^{-1}(a), b\rangle := \mathcal{B}(a, b), \quad \forall a,b \in A. 
\end{equation*}

\begin{lemma}\label{lem:eq4q}
	Let $(A, [\,,\,], \cdot, \mathcal{B})$ be a quadratic Poisson algebra and $I_\mathcal{B}$ be the induced linear isomorphism by $\mathcal{B}$.
	Suppose that $S: A \rightarrow A$ is a linear map that is self-adjoint with respect to $\mathcal{B}$.
	Then the following conditions are equivalent.
	\begin{enumerate}[label=(\roman*)]
		\item\label{it:eq41}
		   $S$ is balanced associated to $(A, \ad_{[,]}, L_\cdot)$.

		\item\label{it:eq42}
		   $S$ is $A$-invariant of mass 1 associated to $(A, \ad_{[,]}, L_\cdot)$.

		\item\label{it:eq43}
		   $P_S = S I_\mathcal{B}: A^* \rightarrow A$ is balanced associated to $(A^*, -\ad_{[,]}^*, L_\cdot^*)$.

		\item\label{it:eq44}
		   $P_S = S I_\mathcal{B}: A^* \rightarrow A$ is $A$-invariant of mass $1$ associated to $(A^*, -\ad_{[,]}^*, L_\cdot^*)$.
	\end{enumerate}
\end{lemma}
\begin{proof}
	Let $x^*, y^* \in A^*$ and $a, b, c \in A$.
	Then the following identities
	\begin{align*}
		 & \mathcal{B}(a, S([b, c]) - [b, S(c)]) = \mathcal{B}([c, S(a)] - [S(c), a], b) = \mathcal{B}(-[S(a), c] - [S(c), a], b),          \\
		 & \mathcal{B}(a, S(b \cdot c) - b \cdot S(c)) = \mathcal{B}(c \cdot S(a) - S(c) \cdot a, b) = \mathcal{B}(S(a) \cdot c - S(c) \cdot a, b),
	\end{align*}
	imply the equivalence between (\ref{it:eq41}) and (\ref{it:eq42}).
	Similarly, the following computation
	\begin{align*}
		& \langle -\ad_{[,]}^*(P_S(x^*))y^* -\ad_{[,]}^*(P_S(y^*))x^*, a\rangle = -\langle y^*, [P_S(x^*), a] \rangle -\langle x^*, [P_S(y^*), a] \rangle                                   \\
		& =-\mathcal{B}(I_\mathcal{B}(y^*), [P_S(x^*), a]) - \mathcal{B}(I_\mathcal{B}(x^*), [P_S(y^*), a]) = -\mathcal{B}([I_\mathcal{B}(y^*), SI_\mathcal{B}(x^*)] + [I_\mathcal{B}(x^*), SI_\mathcal{B}(y^*)], a),      \\
		& \langle L_\cdot^*(P_S(x^*))y^* - L_\cdot^*(P_S(y^*))x^*, a\rangle = \langle y^*, P_S(x^*) \cdot a \rangle - \langle x^*, P_S(y^*) \cdot a \rangle                                   \\
		& = \mathcal{B}(I_\mathcal{B}(y^*), P_S(x^*) \cdot a) - \mathcal{B}(I_\mathcal{B}(x^*), P_S(y^*) \cdot a) = \mathcal{B}(I_\mathcal{B}(y^*) \cdot SI_\mathcal{B}(x^*) - I_\mathcal{B}(x^*) \cdot SI_\mathcal{B}(y^*), a),
	\end{align*}
	establishes the equivalence between (\ref{it:eq41}) and (\ref{it:eq43}).
	Finally, identities
	\begin{align*}
		 & \langle y^*, P_S(-\ad_{[,]}^*(a)x^*) - [a, P_S(x^*)]\rangle = \mathcal{B}(P_S(-\ad_{[,]}^*(a)x^*), I_\mathcal{B}(y^*)) - \mathcal{B}(I_\mathcal{B}(y^*), [a, P_S(x^*)])           \\
		 & = \mathcal{B}(I_\mathcal{B}(-\ad_{[,]}^*(a)x^*), P_S(y^*)) - \mathcal{B}(a, [P_S(x^*), I_\mathcal{B}(y^*)])                                         \\
		 & = \langle (-\ad_{[,]}^*(a)x^*, P_S(y^*) \rangle - \mathcal{B}(a, [P_S(x^*), I_\mathcal{B}(y^*)]) = -\langle x^*, [a, P_S(y^*)]\rangle - \mathcal{B}(a, [P_S(x^*), I_\mathcal{B}(y^*)])   \\
		 & = -\mathcal{B}(a, [SI_\mathcal{B}(y^*), I_\mathcal{B}(x^*)]) - \mathcal{B}(a, [SI_\mathcal{B}(x^*), I_\mathcal{B}(y^*)]),                                  \\
		 & \langle y^*, P_S(L_\cdot^*(a)x^*) - a \cdot P_S(x^*)\rangle = \mathcal{B}(P_S(L_\cdot^*(a)x^*), I_\mathcal{B}(y^*)) - \mathcal{B}(I_\mathcal{B}(y^*), a \cdot P_S(x^*))           \\
		 & = \mathcal{B}(I_\mathcal{B}(L_\cdot^*(a)x^*), P_S(y^*)) - \mathcal{B}(a, P_S(x^*) \cdot I_\mathcal{B}(y^*))                                         \\
		 & = \langle (L_\cdot^*(a)x^*, P_S(y^*) \rangle - \mathcal{B}(a, P_S(x^*) \cdot I_\mathcal{B}(y^*)) = \langle x^*, a \cdot P_S(y^*)\rangle - \mathcal{B}(a, P_S(x^*) \cdot I_\mathcal{B}(y^*)) \\
		 & = \mathcal{B}(a, SI_\mathcal{B}(y^*) \cdot I_\mathcal{B}(x^*)) - \mathcal{B}(a, SI_\mathcal{B}(x^*) \cdot I_\mathcal{B}(y^*)),
	\end{align*}
	show the equivalence between (\ref{it:eq41}) and (\ref{it:eq44}).
	The proof is complete.
\end{proof}

\begin{proposition}\label{prop:ajo2s}
	Let $(A, [\,,\,], \cdot, \mathcal{B})$ be a quadratic Poisson algebra, $I_\mathcal{B}$ be the induced linear isomorphism by $\mathcal{B}$, and $S, T: A \rightarrow A$ be linear maps.
	Then $T$ is an extended $\mathcal{O}$-operator of weight $0$ with extension $S$ of mass $(\kappa, 0)$ on $(A, [\,,\,], \cdot)$ associated to $(A, \ad_{[,]}, L_\cdot)$ if and only if $P_T=T I_\mathcal{B}: A^* \rightarrow A$ is an extended $\mathcal{O}$-operator of weight $0$ with extension $P_S = S I_\mathcal{B}: P^* \rightarrow P$ of mass $(\kappa, 0)$ on $(A, [\,,\,], \cdot)$ associated to $(A^*, -\ad_{[,]}^*, L_\cdot^*)$.	
\end{proposition}
\begin{proof}
	By \cite[Theorem~2.19]{lin2026quasitriangular}, we have
	\begin{equation*}
		I_\mathfrak{B}^{-1} (\ad_{[,]}(a) b) = -\ad_{[,]}^*(a) I_\mathfrak{B}^{-1}(b), \;\;
		I_\mathfrak{B}^{-1} (L_\cdot(a) b) = L_\cdot^*(a) I_\mathfrak{B}^{-1} (b), \;\;
		\forall a, b \in A.
	\end{equation*}
	Therefore,
	\begin{align*}
		 & [P_T(x^*), P_T(y^*)] - P_T\big(-\ad_{[,]}^*(P_T(x^*))y^* + \ad_{[,]}^*(P_T(y^*))x^* \big) - \kappa [P_S(x^*), P_S(y^*)]                                                 \\
		 & =[TI_\mathcal{B}(x^*), TI_\mathcal{B}(y^*)] - TI_\mathcal{B}\big(-\ad_{[,]}^*(TI_\mathcal{B}(x^*))y^* + \ad_{[,]}^*(TI_\mathcal{B}(y^*))x^* \big) - \kappa [SI_\mathcal{B}(x^*), SI_\mathcal{B}(y^*)]          \\
		 & =[TI_\mathcal{B}(x^*), TI_\mathcal{B}(y^*)] - T\big(\ad_{[,]}(TI_\mathcal{B}(x^*))I_\mathcal{B}(y^*) - \ad_{[,]}(TI_\mathcal{B}(y^*))I_\mathcal{B}(x^*) \big) - \kappa [SI_\mathcal{B}(x^*), SI_\mathcal{B}(y^*)], \\
		 & P_T(x^*) \cdot P_T(y^*) - P_T\big(L_\cdot^*(P_T(x^*))t^* + L_\cdot^*(P_T(y^*))x^* \big) - \kappa P_S(x^*) \cdot P_S(y^*)                                                \\
		 & =TI_\mathcal{B}(x^*) \cdot TI_\mathcal{B}(y^*) - TI_\mathcal{B}\big(L_\cdot^*(TI_\mathcal{B}(x^*))y^* + L_\cdot^*(TI_\mathcal{B}(y^*))x^* \big) - \kappa SI_\mathcal{B}(x^*) \cdot SI_\mathcal{B}(y^*)             \\
		 & =TI_\mathcal{B}(x^*) \cdot TI_\mathcal{B}(y^*) - T\big(L_\cdot(TI_\mathcal{B}(x^*))I_\mathcal{B}(y^*) + L_\cdot(TI_\mathcal{B}(y^*))I_\mathcal{B}(x^*) \big) - \kappa SI_\mathcal{B}(x^*) \cdot SI_\mathcal{B}(y^*).
	\end{align*}
	Hence, $T$ is an extended $\mathcal{O}$-operator of weight $0$ with extension $S$ of mass $(\kappa, 0)$ on $(A, [\,,\,], \cdot)$ associated to $(A, \ad_{[,]}, L_\cdot)$ if and only if $P_T=T I_\mathcal{B}: A^* \rightarrow A$ is an extended $\mathcal{O}$-operator of weight $0$ with extension $P_S = S I_\mathcal{B}: P^* \rightarrow P$ of mass $(\kappa, 0)$ on $(A, [\,,\,], \cdot)$ associated to $(A^*, -\ad_{[,]}^*, L_\cdot^*)$.
\end{proof}

\begin{corollary}\label{coro:exoqs}
	Let $(A, [\,,\,], \cdot, \mathcal{B})$ be a quadratic Poisson algebra, $I_\mathcal{B}$ be the induced linear isomorphism by $\mathcal{B}$, and $S, T: A \rightarrow A$ be linear maps.
	Suppose that $S$ is balanced associated to $(A, \ad_{[,]}, L_\cdot)$ and self-adjoint with respect to $\mathcal{B}$, and $T$ is skew-adjoint with respect to $\mathcal{B}$. 
	Then the 2-tensor form of $r_{\pm} = T I_\mathcal{B} \pm S I_\mathcal{B}$ is a solution of the EPYBE of mass $\frac{\kappa+1}{4}$ in $(A, [\,,\,], \cdot)$ if and only if $T$ is an extended $\mathcal{O}$-operator of weight $0$ with extension $S$ of mass $(\kappa, 0)$ on $(A, [\,,\,], \cdot)$ associated to $(A, \ad_{[,]}, L_\cdot)$.
	In particular, the following statements hold.
	\begin{enumerate}[label=(\roman*)]
		\item
		The 2-tensor form of $r_{\pm} = T I_\mathcal{B} \pm S I_\mathcal{B}$ is a solution of the PYBE in $(A, [\,,\,], \cdot)$ if and only if $T$ is an extended $\mathcal{O}$-operator of weight $0$ with extension $S$ of mass $(-1, 0)$ on $(A, [\,,\,], \cdot)$ associated to $(A, \ad_{[,]}, L_\cdot)$.
		
		\item
		The 2-tensor form of $T I_\mathcal{B}$ is a solution of the PYBE in $(A, [\,,\,], \cdot)$ if and only if $T$ is a Rota-Baxter operator of weight $0$ on $(A, [\,,\,], \cdot)$.
	\end{enumerate}
\end{corollary}
\begin{proof}
	By Lemma~\ref{lem:sa2s}, the 2-tensor form of $T I_\mathcal{B}$ is skew-symmetric, while that of $S I_\mathcal{B}$ is symmetric.
	Since $S$ is balanced associated to $(A, \ad_{[,]}, L_\cdot)$, it follows from Lemma~\ref{lem:eq4q} and \ref{lem:sinvb} that the 2-tensor form of $S I_\mathcal{B}$ is $(\ad, L)$-invariant.
	Then, by Theorem~\ref{thm:epybe-o}, the 2-tensor form of $r_{\pm} = T I_\mathcal{B} \pm S I_\mathcal{B}$ is a solution of the EPYBE of mass $\frac{\kappa+1}{4}$ in $(A, [\,,\,], \cdot)$ if and only if $T I_\mathcal{B}: A^* \rightarrow A$ is an extended $\mathcal{O}$-operator of weight $0$ with extension $S I_\mathcal{B}$ of mass $(\kappa, 0)$ on $(A, [\,,\,], \cdot)$ associated to $(A^*, -\ad_{[,]}^*, L_\cdot^*)$.
	Therefore, Proposition~\ref{prop:ajo2s} implies that the 2-tensor form of $r_{\pm}$ is a solution of the EPYBE of mass $\frac{\kappa+1}{4}$ in $(A, [\,,\,], \cdot)$ if and only if $T$ is an extended $\mathcal{O}$-operator of weight $0$ with extension $S$ of mass $(\kappa, 0)$ on $(A, [\,,\,], \cdot)$ associated to $(A, \ad_{[,]}, L_\cdot)$.
	Finally, the particular cases are obtained by taking $\kappa = -1$ and $S = 0$ respectively.
\end{proof}

\begin{corollary}
	Let $(A, [\,,\,], \cdot, \mathcal{B})$ be a quadratic Poisson algebra, $I_\mathcal{B}: A^* \rightarrow A$ be the induced linear isomorphism by $\mathcal{B}$ and $r \in A \otimes A$.
	Writing $r$ as $r = \Theta + \Lambda$ with $\Theta \in \mathrm{Sym}^2(A)$ and $\Lambda \in \mathrm{Alt}^2(A)$.
	Suppose that $\Theta$ is $(\ad, L)$-invariant.
	Then $r$ is a solution of the EPYBE of mass $\frac{\kappa+1}{4}$ if and only if $\Lambda_{+}I_\mathcal{B}^{-1}: A \rightarrow A$ is an extended $\mathcal{O}$-operator of weight $0$ with extension $\Theta_{+}I_\mathcal{B}^{-1}: A \rightarrow A$ of mass $(\kappa, 0)$ on $(A, [\,,\,], \cdot)$ associated to $(A, \ad_{[,]}, L_\cdot)$.
	In particular, the following statements hold.
	\begin{enumerate}[label=(\roman*)]
		\item
		   $r$ is a solution of the PYBE in $(A, [\,,\,], \cdot)$ if and only if $\Lambda_{+}I_\mathcal{B}^{-1}: A \rightarrow A$ is an extended $\mathcal{O}$-operator of weight $0$ with extension $\Theta_{+}I_\mathcal{B}^{-1}: A \rightarrow A$ of mass $(-1, 0)$ on $(A, [\,,\,], \cdot)$ associated to $(A, \ad_{[,]}, L_\cdot)$.

		\item
		   $\Lambda$ is a solution of the PYBE in $(A, [\,,\,], \cdot)$ if and only if $\Lambda_{+}I_\mathcal{B}^{-1}: A \rightarrow A$ is a Rota-Baxter operator of weight 0 on $(A, [\,,\,], \cdot)$.
	\end{enumerate}
\end{corollary}
\begin{proof}
	For all $a, b \in A$, we have
	\begin{align*}
		 & \mathcal{B}(\Lambda_{+}I_\mathcal{B}^{-1}(a), b) + \mathcal{B}(\Lambda_{+}I_\mathcal{B}^{-1}(b), a) = \langle \Lambda_{+}I_\mathcal{B}^{-1}(a), I_\mathcal{B}^{-1}(b) \rangle + \langle \Lambda_{+}I_\mathcal{B}^{-1}(b), I_\mathcal{B}^{-1}(a) \rangle \\
		 & =\langle \Lambda + \tau(\Lambda), I_\mathcal{B}^{-1}(a) \otimes I_\mathcal{B}^{-1}(b)\rangle = 0.
	\end{align*}
	Thus, $\Lambda_{+}I_\mathcal{B}^{-1}$ is skew-adjoint with respect to $\mathcal{B}$.
	Similarly, $\Theta_{+}I_\mathcal{B}^{-1}$ is self-adjoint with respect to $\mathcal{B}$.
	Moreover, since $\Theta$ is $(\ad, L)$-invariant, it follows from Lemma~\ref{lem:sinvb} and \ref{lem:eq4q} that $\Theta_{+}I_\mathcal{B}^{-1}:A \rightarrow A$ is balanced associated to $(A, \ad_{[,]}, L_\cdot)$.
	The desired conclusions now follow from Corollary~\ref{coro:exoqs}.
\end{proof}

\subsection{Extended \texorpdfstring{$\mathcal{O}$}{O}-operators and EPYBE on semi-direct product Poisson algebras}
\label{sec:extension}

\begin{lemma}\label{lem:b2bsm}
	Let $(A, [\,,\,], \cdot)$ be a Poisson algebra, $(V, \rho, \zeta)$ a module of $(A, [\,,\,], \cdot)$, and $(\mathfrak{A}=A \ltimes_{-\rho^*, \zeta^*} V^*, \courant{\,,\,}, \bullet)$ be the semi-direct product of $(A, [\,,\,], \cdot)$ and $(V^*, -\rho^*, \zeta^*)$.
	Let $S: V \rightarrow A$ be a linear map.
	Then $\mathcal{S} := r^{S}_{+} - r^{S}_{-}: \mathfrak{A}^* \rightarrow \mathfrak{A}$ is a balanced $\mathfrak{A}$-module homomorphism associated to $(\mathfrak{A}^*, -\ad_{\courant{,}}^*, L_\bullet^*)$ if and only if $S$ is a balanced $A$-module homomorphism associated to $(V, \rho, \zeta)$.
\end{lemma}
\begin{proof}
	Note that $r^S+(\tau(r^S))$ is symmetric, by Lemma~\ref{lem:sinvb},
	$r^{S}_{+} - r^{S}_{-} = (r^S+ \tau(r^S) )_{+}: \mathfrak{A}^* \rightarrow \mathfrak{A}$ is a balanced $\mathfrak{A}$-module homomorphism associated to $(\mathfrak{A}^*, -\ad_{\courant{,}}^*, L_\bullet^*)$ if and only if it is balanced associated to $(\mathfrak{A}^*, -\ad_{\courant{,}}^*, L_\bullet^*)$.
	It then suffices to prove that $\mathcal{S}$ is balanced associated to $(\mathfrak{A}^*, -\ad_{\courant{,}}^*, L_\bullet^*)$ if and only if $S$ is a balanced $A$-module homomorphism associated to $(V, \rho, \zeta)$.
	For all $x^* \in A^*$ and $v \in V$, we have
	\begin{equation*}
		(r^{S}_{+} - r^{S}_{-})(x^* + u) = S(u) + S^*(x^*).
	\end{equation*}
	Let $x^*, y^* \in A^*$, $a \in A$, $\xi^* \in V^*$ and $u, v \in V$, we have
	\begin{align*}
		 & \langle \ad_{\courant{,}}^*((r^{S}_{+} - r^{S}_{-})(x^* + u)) (y^* + v) + \ad_{\courant{,}}^*((r^{S}_{+} - r^{S}_{-})(y^* + u)) (x^* + v), a + \xi^*\rangle                 \\
		 & = \langle \ad_{\courant{,}}^*(S(u) + S^*(x^*)) (y^* + v) + \ad_{\courant{,}}^*(S(v) + S^*(y^*)) (x^* + u), a + \xi^* \rangle                                \\
		 & = \langle y^* + v, \courant{S(u) + S^*(x^*), a + w^*}\rangle + \langle x^* + u, \courant{S(v) + S^*(y^*), a + \xi^*}\rangle                                 \\
		 & = \langle y^*, [S(u), a]\rangle + \langle x^*, [S(v), a]\rangle - \langle v, \rho^*(S(u))\xi^* \!-\! \rho^*(a)S^*(x^*)\rangle - \langle u, \rho^*(S(v))\xi^* \!-\! \rho^*(a)S^*(y^*)\rangle \\
		 & = \langle y^*, [S(u), a] + S(\rho(a)u) \rangle + \langle x^*, [S(v), a] + S(\rho(a)v)\rangle - \langle \xi^*, \rho(S(u))v + \rho(S(v))u \rangle,
	\end{align*}
	and similarly
	\begin{align*}
		 & \langle -L_{\bullet}^*((r^{S}_{+} - r^{S}_{-})(x^* + u)) (y^* + v) + L_{\bullet}^*((r^{S}_{+} - r^{S}_{-})(y^* + u)) (x^* + v), a + \xi^*\rangle \\
		 & = \langle -L_{\bullet}^*(S(u) + S^*(x^*)) (y^* + v) + L_{\bullet}^*(S(v) + S^*(y^*)) (x^* + u), a + \xi^* \rangle                \\
		 & = -\langle y^* + v, (S(u) + S^*(x^*)) \bullet (a + \xi^*)\rangle + \langle x^* + u, (S(v) + S^*(y^*)) \bullet (a + \xi^*) \rangle        \\
		 & = -\langle y^*, S(u) \cdot a\rangle + \langle x^*, S(v) \cdot a \rangle - \langle v,
		\zeta^*(S(u))\xi^* \!+\! \zeta^*(a)S^*(x^*)\rangle + \langle u, \zeta^*(S(v))\xi^* \!+\! \zeta^*(a)S^*(y^*)\rangle                 \\
		 & = \langle y^*, -S(u) \cdot a + S(
		\zeta(a)u) \rangle + \langle x^*, S(v) \cdot a - S(\zeta(a)v)\rangle + \langle \xi^*, -\zeta(S(u))v + \zeta(S(v))u \rangle.
	\end{align*}
	Therefore, $r^{S}_{+} - r^{S}_{-}: \mathfrak{A}^* \rightarrow \mathfrak{A}$ is balanced associated to $(\mathfrak{A}^*, -\ad_{\courant{,}}^*, L_\bullet^*)$ if and only if $S$ is a balanced $A$-module homomorphism associated to $(V, \rho, \zeta)$.
	The proof is complete.
\end{proof}

\begin{theorem}\label{thm:eo2smp}
	Let $(A, [\,,\,], \cdot)$ be a Poisson algebra, $(V, \rho, \zeta)$ be a module of $(A, [\,,\,], \cdot)$, and $(\mathfrak{A}=A \ltimes_{-\rho^*, \zeta^*} V^*, \courant{\,,\,}, \bullet)$ be the semi-direct product of $(A, [\,,\,], \cdot)$ and $(V^*, -\rho^*, \zeta^*)$.
	Let $S, T: V \rightarrow A$ be linear maps.
	Suppose that $S$ is a balanced $A$-module homomorphism associated to $(V, \rho, \zeta)$.
	Then  the following conditions are equivalent.
		\begin{enumerate}[label=(\roman*)]
		 \item\label{SDJ1} $(r^T-\tau(r^T)) \pm (r^S + \tau(r^S))$ is a solution of the EPYBE of mass $\frac{\kappa+1}{4}$ in $(A \ltimes_{-\rho^*, \zeta^*} V^*, \courant{\,,\,}, \bullet)$. 
		\item\label{SDJ2} $\mathcal{T} := r_{+}^{T} + r_{-}^{T}: \mathfrak{A}^* \rightarrow \mathfrak{A}$ is an extended $\mathcal{O}$-operator of weight 0 with extension $\mathcal{S} := r_{+}^{S} - r_{-}^{S}$ of mass $(\kappa, 0)$ on $(\mathfrak{A}, \courant{\,,\,}, \bullet)$ associated $(\mathfrak{A}^*, -\ad_{\courant{\,,\,}}^*, L_\bullet^*)$.
				\item\label{SDJ3}  $T: V \rightarrow A$ is an extended $\mathcal{O}$-operator of weight 0 with extension $S$ of mass $(\kappa, 0)$ on $(A, [\,,\,], \cdot)$ associated to $(V, \rho, \zeta)$.
	 \end{enumerate}
\end{theorem}
\begin{proof}
	\ref{SDJ1} $\Longleftrightarrow$ \ref{SDJ2}. 
	By Lemmas~\ref{lem:sinvb} and \ref{lem:b2bsm}, $r^S + \tau(r^S)$ is $(\ad, L)$-invariant.
	Applying Theorem~\ref{thm:epybe-o}, we get that the 2-tensor $(r^T-\tau(r^T)) \pm (r^S + \tau(r^S))$, with the skew-symmetric part $r^T-\tau(r^T)$ and symmetric part $r^S + \tau(r^S)$, 
	is a solution of the EPYBE of mass $\frac{\kappa+1}{4}$ in $(\mathfrak{A}, \courant{\,,\,}, \bullet)$ if and only if $\mathcal{T}: \mathfrak{A}^* \rightarrow \mathfrak{A}$ is an extended $\mathcal{O}$-operator of weight 0 with extension $\mathcal{S}$ of mass $(\kappa, 0)$ on $(\mathfrak{A}, \courant{\,,\,}, \bullet)$ associated $(\mathfrak{A}^*, -\ad_{\courant{\,,\,}}^*, L_\bullet^*)$.

	\ref{SDJ2} $\Longleftrightarrow$ \ref{SDJ3}. 
	For all $x^* \in A^*$ and $u \in V$, we have
	\begin{equation*}
		\mathcal{T}(x^* + u) = - T(u) + T^*(x^*).
	\end{equation*}
	For all $x^*, y^*, z^* \in A^*$ and $u, v, w \in V$, we have
	\begin{align*}
		 & \langle \courant{\mathcal{T}(x^* + u), \mathcal{T}(y^* + v)} - \kappa \courant{\mathcal{S}(x^*+u), \mathcal{S}(y^*+v)}, z^*+ w\rangle               \\
		 & \quad - \langle \mathcal{T}\big( -\ad_{\courant{\,,\,}}^*(\mathcal{T}(x^* + u))(y^*+v) + \ad_{\courant{\,,\,}}^*(\mathcal{T}(y^* + v))(x^*+u)\big), z^*+ w\rangle \\
		 & =\langle \courant{-T(u) + T^*(x^*), - T(v) + T^*(y^*) } - \kappa \courant{S(u) + S^*(x^*), S(v) + S^*(y^*)}, z^*+ w\rangle                      \\
		 & \quad - \langle \mathcal{T}\big( -\ad_{\courant{\,,\,}}^*(-T(u) + T^*(x^*))(y^*+v) + \ad_{\courant{\,,\,}}^*(- T(v) + T^*(y^*))(x^*+u)\big), z^*+ w\rangle      \\
		 & =\langle z^*, [T(u), T(v)] - \kappa [S(u), S(v)] \rangle                                                     \\
		 & \quad + \langle \rho^*(T(u))T^*(y^*) - \rho^*(T(v))T^*(x^*) + \kappa \rho^*(S(u))S^*(y^*) - \kappa \rho^*(S(v))S^*(x^*), w\rangle                 \\
		 & \quad +\langle -\ad_{\courant{\,,\,}}^*(- T(u) + T^*(x^*) )(y^*+v) + \ad_{\courant{\,,\,}}^*(- T(v) + T^*(y^*) )(x^*+u), \mathcal{T}(z^* + w)\rangle          \\
		 & =\langle z^*, [T(u), T(v)] - \kappa [S(u), S(v)]\rangle                                                      \\
		 & \quad + \langle y^*, T(\rho(T(u))w) + \kappa S(\rho(S(u))w)\rangle - \langle x^*, T(\rho(T(v))w) + \kappa S(\rho(S(v))w) \rangle                 \\
		 & \quad - \langle y^* + v, \courant{ - T(u) + T^*(x^*), - T(w) + T^*(z^*) } \rangle + \langle x^* + u, \courant{- T(v) + T^*(y^*), - T(w) + T^*(z^*) } \rangle          \\
		 & =\langle z^*, [T(u), T(v)] - \kappa [S(u), S(v)]\rangle                                                      \\
		 & \quad + \langle y^*, T(\rho(T(u))w) + \kappa S(\rho(S(u))w)\rangle - \langle x^*, T(\rho(T(v))w) + \kappa S(\rho(S(v))w) \rangle                 \\
		 & \quad - \langle y^* + v, [T(u), T(w)] + \rho^*(T(u))(T^*(z^*)) - \rho^*(T(w))(T^*(x^*)) \rangle                                  \\
		 & \quad + \langle x^* + u, [T(v), T(w)] + \rho^*(T(v))(T^*(z^*)) - \rho^*(T(w))(T^*(y^*)) \rangle                                  \\
		 & =\langle z^*, [T(u), T(v)] - T(\rho(T(u))v) + T(\rho(T(v))u) - \kappa [S(u), S(v)]\rangle                                     \\
		 & \quad + \langle y^*, T(\rho(T(u))w) - T(\rho(T(w))u) - [T(u), T(w)] + \kappa S(\rho(S(u))w) \rangle                                \\
		 & \quad - \langle x^*, T(\rho(T(v))w) - T(\rho(T(w))v) -[T(v), T(w)] + \kappa S(\rho(S(v))w) \rangle                                \\
		 & =\langle z^*, [T(u), T(v)] - T(\rho(T(u))v) + T(\rho(T(v))u) - \kappa [S(u), S(v)] \rangle \\
		 & \quad + \langle y^*, - [T(u), T(w)] + T(\rho(T(u))w) - T(\rho(T(w))u) + \kappa [S(u), S(w)] \rangle \\
		 & \quad - \langle x^*, -[T(v), T(w)] + T(\rho(T(v))w) - T(\rho(T(w))v) + \kappa [S(u), S(w)] \rangle,
	\end{align*}
	and similarly
	\begin{align*}
		 & \langle \mathcal{T}(x^* + u) \bullet \mathcal{T}(y^* + v) - \kappa \mathcal{S}(x^*+u) \bullet \mathcal{S}(y^*+v), z^*+ w\rangle     \\
		 & \quad - \langle \mathcal{T}\big( L_\bullet^*(\mathcal{T}(x^* + u))(y^*+v) + L_\bullet^*(\mathcal{T}(y^* + v))(x^*+u)\big), z^*+ w\rangle \\
		 & =\langle z^*, T(u) \cdot T(v) - T(\zeta(T(u))v) - T(\zeta(T(v))u) - \kappa S(u) \cdot S(v) \rangle                    \\
		 & \quad + \langle y^*, - T(u) \cdot T(w) + T(\zeta(T(u))w) + T(\zeta(T(w))u) + \kappa S(u) \cdot S(w) \rangle               \\
		 & \quad - \langle x^*, - T(v) \cdot T(w) + T(\zeta(T(v))w) + T(\zeta(T(w))v) + \kappa S(u) \cdot S(w) \rangle.
	\end{align*}
	Therefore, $\mathcal{T}$ is an extended $\mathcal{O}$-operator of weight 0 with extension $\mathcal{S}$ of mass $(\kappa, 0)$ on $(\mathfrak{A}, \courant{\,,\,}, \bullet)$ associated $(\mathfrak{A}^*, -\ad_{\courant{\,,\,}}^*, L_\bullet^*)$ if and only if $T: V \rightarrow A$ is an extended $\mathcal{O}$-operator of weight 0 with extension $S$ of mass $(\kappa, 0)$ on $(A, [\,,\,], \cdot)$ associated to $(V, \rho, \zeta)$.
\end{proof}

\begin{corollary}\label{cor:tsc} 
	 With the conditions in Theorem \ref{thm:eo2smp}, the following statements hold.
	\begin{enumerate}[label=(\roman*)]
		\item\label{it:tcts1}
		   $T$ is an extended $\mathcal{O}$-operator of weight 0 with extension $S$ of mass $(-1, 0)$ on $(A, [\,,\,], \cdot)$ associated to $(V, \rho, \zeta)$ if and only if $(r^T-\tau(r^T)) \pm (r^S + \tau(r^S))$ is a solution of the PYBE in $(A \ltimes_{-\rho^*, \zeta^*} V^*, \courant{\,,\,}, \bullet)$.

		\item
		   $T$ is an $\mathcal{O}$-operator of weight 0 associated to $(V, \rho, \zeta)$ if and only if $r^T-\tau(r^T)$ is a skew-symmetric solution of the PYBE in $(A \ltimes_{-\rho^*, \zeta^*} V^*, \courant{\,,\,}, \bullet)$.
	\end{enumerate}
\end{corollary}
\begin{proof}
	It follows form Theorem \ref{thm:eo2smp} by taking $\kappa = -1$ or $S = 0$.
\end{proof}

\begin{corollary}\label{cor:tscx}
	Let $(A, [\,,\,], \cdot)$ be a Poisson algebra and $(\mathfrak{A} = A \ltimes_{-\ad_{[,]}^*, L_\cdot^*} A^*, \courant{\,,\,}, \bullet)$ be the semi-direct product of $(A, [\,,\,], \cdot)$ and $(A^*, -\ad_{[,]}^*, L_\cdot^*)$,
	and $T: A \rightarrow A$ be a linear map.
	Then
	\begin{enumerate}[label=(\roman*)]
		\item\label{it:c1}
		   $T$ is an extended $\mathcal{O}$-operator of weight 0 with extension $\id$ of mass $(-1,0)$ on $(A, [\,,\,], \cdot)$ associated to $(V, \rho, \zeta)$ if and only if $r^T-\tau(r^T) \pm (r^\id + \tau(r^\id))$ is a solution of the PYBE in $(A \ltimes_{-\ad_{[,]}^*, L_\cdot^*} A^*, \courant{\,,\,}, \bullet)$.

		\item\label{it:c2}
		   $T$ is a Rota-Baxter operator of nonzero weight $\lambda$ on $(A, [\,,\,], \cdot)$ if and only if both $(r^T-\tau(r^T)) + \lambda r^\id$ and $(r^T-\tau(r^T)) - \lambda\tau(r^\id)$ are solutions of the PYBE in $(A \ltimes_{-\ad_{[,]}^*, L_\cdot^*} A^*, \courant{\,,\,}, \bullet)$.
	\end{enumerate}
\end{corollary}
\begin{proof}
	\ref{it:c1}. 
	The conclusion follows from Corollary~\ref{cor:tsc} \ref{it:tcts1} by taking $S = \id$.
	
	\ref{it:c2}.
	Note that $T: A \to A$ is a Rota-Baxter operator of weight $\lambda$ on $(A, [\,,\,], \cdot)$ if and only if $T + \frac{\lambda}{2}\id$ is an extended $\mathcal{O}$-operator of weight $0$ with extension $\frac{\lambda}{2}\id$ of mass $(-1, 0)$ on $(A, [\,,\,], \cdot)$ associated to $(A, \ad_{[,]}, L_{\cdot})$.
	Therefore, the conclusion follows from Corollary~\ref{cor:tsc} \ref{it:c1}.
\end{proof}

\begin{example}
	Let $(A, [\,,\,], \cdot)$ be the 3-dimensional Poisson algebra given in Example~\ref{HesPois} with basis $\{e_1,e_2,e_3\}$ whose non-zero bracket $[\,,\,]$ and multiplication $\cdot$ are given by Eq.~\eqref{HPCF1}, and
	$P: A \to A$ be the Rota-Baxter operator of weight $1$ defined by Eq.~\eqref{HPCF1RB}.
	Let $\{e_1^*, e^*_2, e_3^*\}$ be the dual basis of $\{e_1,e_2,e_3\}$. Then Corollary~\ref{cor:tscx} \ref{it:c2} shows that both 
	\begin{align*}
		r^T - \tau(r^T) + r^\id &=  2e_1 \otimes e_1^* + 3e_2 \otimes e_2^* + \frac{3}{2}e_3 \otimes e_3^*  - e_1^* \otimes e_1 - 2e_2^* \otimes e_2 - \frac{1}{2}e_3^* \otimes e_3,\\	
				r^T - \tau(r^T) - \tau(r^\id) &=  e_1 \otimes e_1^* + 2e_2 \otimes e_2^* + \frac{1}{2}e_3 \otimes e_3^* - 2e_1^* \otimes e_1 - 3e_2^* \otimes e_2 - \frac{3}{2}e_3^* \otimes e_3,	
	\end{align*}
	are solutions of the PYBE in $(A \ltimes_{-\ad_{[,]}^*, L_\cdot^*} A^*, \courant{\,,\,}, \bullet)$, where non-zero bracket $\courant{,}$ and multiplication $\bullet$ are given Eq.~\eqref{HesPoisSP}.
\end{example}


\bigskip




\begin{thebibliography}{99}
	
\bibitem{arnol2013mathematical}
V.~I. Arnol'd, \emph{Mathematical methods of classical mechanics}, Springer Science \& Business Media, 2013.

\bibitem{bai2013splitting}
C.~Bai, O.~Bellier, L.~Guo and X.~Ni, Splitting of operations, Manin products and Rota-Baxter operators, \emph{Int. Math. Res. Not.} 3 (2013), 485--524.


\bibitem{bai2010nonabelian}
C.~Bai, L.~Guo and X.~Ni, Nonabelian generalized {Lax} pairs, the classical {Yang-Baxter} equation and {post-Lie} algebras, \emph{Commun. Math. Phys.} 297 (2010), 553--596.

\bibitem{bai2011generalizations}
C.~Bai, L.~Guo and X.~Ni, Generalizations of the classical {Yang-Baxter} equation and $\mathcal{O}$-operators, \emph{J. Math. Phys.} 52 (2011).

\bibitem{bai2012O}
C.~Bai, L.~Guo and X.~Ni, $\mathcal{O}$-operators on associative algebras and associative {Yang-Baxter} equations, \emph{Pacific J. Math.} 256 (2012), 257--289.

\bibitem{chari1995guide}
V.~Chari and A.~Pressley, \emph{A guide to quantum groups}, Cambridge University Press, 1995.

\bibitem{chen2024postpoisson}
S.~Chen, C.~Bai and L.~Guo, Solving the Poisson Yang-Baxter equation via deformation-to-quasiclassical-limits, \emph{arXiv e-prints}, arXiv:2404.11232v2, 2024.

\bibitem{dirac2013lectures}
P.~A. Dirac, \emph{Lectures on quantum mechanics}, Courier Corporation, 2013.

\bibitem{dotsenko2021endofunctors}
V.~Dotsenko and P.~Tamaroff, Endofunctors and Poincaré-Birkhoff-Witt Theorems, \emph{Int. Math. Res. Not.} 16 (2021), 12670--12690

\bibitem{drinfeld1986quantum}
V.~Drinfeld, \emph{Quantum groups}, Proc. Int. Congr. Math., 1986, 798--820.

\bibitem{drinfeld1983hamiltonian}
V.~G. Drinfeld, Hamiltonian structures on {Lie} groups, {Lie} bialgebras and the geometric meaning of the classical {Yang-Baxter} equations, \emph{Soviet Math. Dokl.} 27 (1983), 68--71.

\bibitem{ginzburg2004poisson}
V.~Ginzburg and D.~Kaledin, Poisson deformations of symplectic quotient singularities, \emph{Adv. Math.} 186 (2004), 1--57.

\bibitem{huebschmann1990poisson}
J.~Huebschmann, Poisson cohomology and quantization, \emph{J. Reine Angew. Math.} 408 (1990), 57--113.

\bibitem{kontsevich2003deformation}
M.~Kontsevich, Deformation quantization of {Poisson} manifolds, \emph{Lett. Math. Phys.} 66 (2003), 157--216.

\bibitem{kupershmidt1999classical}
B.~A. Kupershmidt, What a classical $r$-matrix really is, \emph{J. Nonlinear Math. Phys.} 6 (1999), 448--488.

\bibitem{lin2023differential}
Y.~Lin, X.~Liu and C.~Bai, Differential antisymmetric infinitesimal bialgebras, coherent derivations and Poisson bialgebras, \emph{Symmetry Integrability Geom. Methods Appl.} 19 (2023), 018.

\bibitem{lin2026quasitriangular}
Y.~Lin and D.~Lu, Quasitriangular and factorizable {Poisson} bialgebras, \emph{Pacific J. Math.} 343 (2026), 453--483.



 
 \bibitem{LBS2020}
 J. Liu, C. Bai and Y. Sheng, Noncommutative Poisson bialgebras, \emph{ J. Algebra} 556 (2020), 35--66. 

\bibitem{loday2007algebra}
J.-L. Loday, On the algebra of quasi-shuffles, \emph{Manuscripta Math.} 123 (2007), 79--93.

\bibitem{loday2004trialgebras}
J.-L. Loday and M.~Ronco, Trialgebras and families of polytopes, \emph{Contemp. Math.} 346 (2004), 369--398.

\bibitem{ni2013poisson}
X.~Ni and C.~Bai, Poisson bialgebras, \emph{J. Math. Phys.} 54 (2013), 023515.

\bibitem{odzijewicz2011hamiltonian}
A.~Odzijewicz, Hamiltonian and quantum mechanics, \emph{Geom. Topol. Monogr} 17 (2011), 385--472.

\bibitem{ospel2022polarization}
C.~Ospel, F.~Panaite and P.~Vanhaecke, Polarization and deformations of generalized dendriform algebras, \emph{J. Noncommut. Geom.} 16 (2022), 561--594.

\bibitem{polishchuk1997algebraic}
A.~Polishchuk, Algebraic geometry of {Poisson} brackets, \emph{J. Math. Sci.} 84 (1997), 1413--1444.

\bibitem{semenov1983classical}
M.~Semenov-Tyan-Shanskii, What is a classical $r$-matrix?, \emph{Funct. Anal. Appl.} 17 (1983), 259--272.

\bibitem{vaisman2012lectures}
I.~Vaisman, \emph{Lectures on the geometry of {Poisson} manifolds}, Birkh{\"a}user, 2012.

\bibitem{vallette2007homology}
B.~Vallette, Homology of generalized partition posets, \emph{J. Pure Appl. Algebra} 208 (2007), 699--725.

\bibitem{vallette2007Maninp}
B.~Vallette, Manin products, Koszul duality, Loday algebras and Deligne conjecture, \emph{J. reine angew. Math.} 620 (2008), 105--164.



\bibitem{weinstein1977lectures}
A.~Weinstein, \emph{Lectures on symplectic manifolds}, American Mathematical Soc., 1977.

\bibitem{yu2026extended}
J.~Yu and Y.~Hong, Extended $\mathcal{O}$-operators, {Novikov} {Yang-Baxter} equations and {post-Novikov} algebras, \emph{J. Nonlinear Math. Phys.} 33 (2026), 14.

\end{thebibliography}

\end{document}